\documentclass{amsart}
\usepackage{amssymb, amscd, enumitem}
\usepackage{hyperref}
\usepackage{pinlabel}
\usepackage{mathtools}

\newtheorem{theorem}{Theorem} [section]

\newtheorem{lemma}[theorem]{Lemma}
\newtheorem{proposition}[theorem]{Proposition}
\newtheorem{corollary}[theorem]{Corollary}

\newtheorem{remark}[theorem]{Remark}

\newtheorem{theoremalpha}{Theorem}
\theoremstyle{definition}
\newtheorem{definition}[theorem]{Definition}
\newtheorem{notation}[theorem]{Notation}

\newcommand{\wt}[1]{\widetilde{#1}}

\newcommand\N{{\mathbb N}}
\newcommand\Z{{\mathbb Z}}

\newcommand{\cE }{\mathcal E}

\newcommand{\cL }{\mathcal L}

\newcommand{\cQ }{\mathcal Q}
\newcommand{\cP }{\mathcal P}

\newcommand{\cS }{\mathcal S}
\newcommand{\cT }{\mathcal T}

\newcommand{\p}{\mathrm{patt}}
\newcommand{\horz}{\rule[2.2pt]{2em}{0.4pt}}

\newcommand{\bs}[1]{\boldsymbol{#1}}

\newcommand{\w}{\vec{w}}
\renewcommand{\v}{\vec{v}}
\renewcommand{\u}{\vec{u}}

\address{Department of Mathematics\\
	Heriot-Watt University\\
	Edinburgh\\
	Scotland\\
	EH14 4AS}

\email{ace2@hw.ac.uk}

\begin{document}

\author{Alex Evetts}

\title{Rational Growth in Virtually Abelian Groups}

\begin{abstract}
We show that any subgroup of a finitely generated virtually abelian group $G$ grows rationally relative to $G$, that the set of right cosets of any subgroup of $G$ grows rationally, and that the set of conjugacy classes of $G$ grows rationally. These results hold regardless of the choice of finite weighted generating set for $G$. 

\end{abstract}
\maketitle

%%%%%%%%%%%%%%%%%%%%%%%%%%%%%%%%%%%%%%%%%%%%%%%%%%%%%%%%%%%%%%%%%%%%%%%%%%%%
\section{Introduction}\label{sec:intro}
%%%%%%%%%%%%%%%%%%%%%%%%%%%%%%%%%%%%%%%%%%%%%%%%%%%%%%%%%%%%%%%%%%%%%%%%%%%%
%%%%%%%%%%%%%%%%%%%%%%%%%%%%%%%%%%%%%%%%%%%%%%%%%%%%%%%%%%%%%%%%%%%%%%%%%%%%

The related notions of growth functions and growth series of finitely generated groups have attracted a lot of attention in many different classes of groups \cite{DuchinSurvey}. In the case of growth series, it is perhaps surprising that there are very few results that are independent of the choice of finite generating set. In this paper we add three facts to those results. Namely that for a virtually abelian group, the conjugacy growth series, the relative growth series of any subgroup, and the coset growth series with respect to any subgroup, are rational functions for any choice of generating set.

The standard weighted growth series of any virtually abelian group was shown by Benson to be rational in \cite{Benson}. Liardet \cite{Liardet} built on work of Klarner \cite{Klarner} and others to show that the complete growth series (a generalisation of standard growth) is also rational. The corresponding result (with weight function uniformly equal to $1$) was proved for hyperbolic groups in the 1980s (see \cite{CannonNote}, \cite{CannonPaper}, \cite{WordProcessing}, and \cite{GromovEssays}), and for the integer Heisenberg group in \cite{DuchinShapiro}. The significance of these results comes from the fact that they hold for any choice of generating set $S$. This need not be the case in general. Stoll showed in \cite{Stoll} that there exist groups whose growth series are rational with respect to one generating set and transcendental with respect to another. Many other groups have been shown to have rational growth series with respect to `standard' generators, for example Coxeter groups, the soluble Baumslag-Solitar groups, certain automatic groups. However, little is known for other generating sets, or other types of growth series.

Let $G$ be a group with a finite generating set $S$, that generates $G$ as a monoid. Let $S^*$ denote the set of all words over $S$ (i.e. the free monoid on $S$). If we assign a positive integer weight, $\omega(s)$, to each element $s\in S$, we can define the weight of any word $s_1s_2\cdots s_k\in S^*$ to be $\sum_{i=1}^k\omega(s_i)$. In turn, we define the weight of an element $g\in G$, denoted $\omega(g)$, to be the minimal weight amongst all words that represent $g$.

We define the $\emph{(standard) weighted growth function}$ of $G$ with respect to $S$ and $\omega$ to be \begin{equation}\label{eq:standardfunction}
\sigma^\omega_{G,S}(n)=\#\{g\in G\mid\omega(g)=n\}.
\end{equation}

The \emph{weighted growth series} of a group $G$, with respect to a finite generating set $S$ with weight function $\omega$ is defined as
	\begin{equation}\label{eq:standardseries}
	\cS^\omega_{G,S}(z)=\sum_{n=0}^\infty\sigma^\omega_{G,S}(n)z^n.
	\end{equation}

It is natural to ask under what conditions $\cS^\omega_{G,S}(z)$ can be expressed as a rational function, that is, when do there exist polynomials $p,q$ with integer coefficients such that $\cS^\omega_{G,S}(z)=\frac{p(z)}{q(z)}$. The definition given via \eqref{eq:standardfunction} and \eqref{eq:standardseries} is often referred to as \emph{strict} or \emph{spherical} growth. One can also consider \emph{cumulative} growth, which instead counts elements with weight \emph{at most} $n$. Note that the spherical growth series is rational if and only if the cumulative growth series is rational.

In this paper we focus on virtually abelian groups, and study various generalisations of the notion of growth series. In each case, the results will hold for any choice of weighted generating set.

Let $G$ be a group generated by a finite set $S$, with weight function $\omega$, and consider a subgroup $H$. Let $\sigma^\omega_{H\leq G,S}(n)$ be the number of elements of $H$ which have weight $n$ with respect to the generators of $G$. We will call this the \emph{weighted growth function of $H$ relative to $G$}. The \emph{weighted growth series of $H$ relative to $G$} is then 
\begin{equation}
\cS^\omega_{H\leq G,S}(z)=\sum_{n=0}^\infty\sigma^\omega_{H\leq G,S}(n)z^n.
\end{equation}

The asymptotic behaviour of the relative growth function $\sigma^\omega_{H\leq G,S}(n)$ has been studied in many papers (for example \cite{DavisOlshanskii}). In the present paper we consider the much less studied formal power series. In Section \ref{sec:Relative} we prove the following.
\begin{theoremalpha}[See Theorem \ref{thm:relative}]
	If $G$ is virtually abelian, and $H\leq G$ is any subgroup, then the weighted growth series of $H$ relative to $G$ is rational with respect to any finite generating set $S$ of $G$.
\end{theoremalpha}

One way to view standard growth is to consider the equivalence relation on $S^*$ where words are equivalent whenever they represent the same group element. We then choose one minimal weight word from each equivalence class, and study the growth of the resulting language. Any other equivalence relation on $S^*$ thus yields a growth function (and corresponding series) in the same way.

In Section \ref{sec:coset}, we study the growth series arising from the equivalence relation where words are considered to be equivalent whenever they represent elements of the same coset of a chosen subgroup $H$. This is known as \emph{coset growth}. If the weight of every generator is $1$, then the coset growth function counts the number of cosets intersecting the sphere of radius $n$ in the Cayley graph (that do not intersect the sphere of radius $n-1$). Coset growth has been studied for hyperbolic groups by Holt and others. See \cite{CosetAuto} for details. Here, we prove the following.
\begin{theoremalpha}[See Theorem \ref{thm:cosetfull}]
	If $G$ is virtually abelian, and $H\leq G$ is any subgroup, then the weighted growth series of the set of right cosets $H\backslash G$ is rational with respect to any finite generating set $S$ of $G$.
\end{theoremalpha}

In Section \ref{sec:conjugacy} we study the growth series arising from the equivalence relation where words are considered to be equivalent whenever they represent elements of the same conjugacy class. This is known as \emph{conjugacy growth}, and has been studied for some time, for example in \cite{CoornaertKnieper} and \cite{SolvableConjugacy}. To the author's knowledge, Babenko \cite{Babenko} was the first to introduce this form of growth. For a useful overview see Guba and Sapir's paper \cite{GubaSapir}. Rivin studied the corresponding formal power series, and conjectured (\cite{Rivin1}, \cite{Rivin2}) that a hyperbolic group has rational conjugacy growth series if and only if it is virtually cyclic. One direction of this was confirmed by Ciobanu, Hermiller, Holt, and Rees in \cite{CHHR}, and the other by Antol\'in and Ciobanu in \cite{FormalConjugacyGrowth}. Recently, Gekhtman and Yang (\cite{ContractingElements}) have shown that the conjugacy growth series is transcendental for all (non-elementary) relatively hyperbolic groups, and a large class of acyclindrically hyperbolic groups, with respect to any choice of finite generating set.

Amongst other calculations, Mercier \cite{Mercier} has shown that the conjugacy growth series of the lamplighter group $C_2\wr\Z$ is transcendental with respect to a certain choice of generating set. Section 5 of \cite{CHHR} contains calculations of conjugacy growth and related series in some virtually abelian groups, and \cite{CHM} contains formulas for the conjugacy growth series of graph products.

In the present paper we prove the following.
\begin{theoremalpha}[See Theorem \ref{thm:conj}]\label{ThmC}
	If $G$ is virtually abelian, then the weighted conjugacy growth series of $G$ is rational with respect to any choice of finite generating set $S$.
\end{theoremalpha}

In light of the results mentioned above, we venture the conjecture that the conjugacy growth series of a finitely presented group that is not virtually abelian is transcendental, with respect to all finite generating sets.

Our key tool for all of the results will be the theory of polyhedral sets developed in \cite{Benson}. These are defined to be subsets of $\Z^r$ for some $r>0$, made up of regions bounded by affine hyperplanes. Assigning a weight to each coordinate of $\Z^r$, and then extending this linearly to a weight function on the whole of $\Z^r$, we can calculate the growth series of any subset. It turns out that polyhedral sets have rational growth series (see Section \ref{sec:polyhedral} for definitions and results). Our strategy will always be to find a language (i.e. a subset of $S^*$) consisting of minimal representatives for the objects we wish to count, which is in one-to-one weight-preserving correspondence with a polyhedral set, and thus has rational growth series. The objects in question may be group elements, cosets, or conjugacy classes.

In \cite{Benson}, it is shown that there exists a language of minimal representatives for the elements of a given virtually abelian group, which is in one-to-one weight-preserving correspondence with a finite collection of polyhedral sets. Thus the (standard) growth series is a rational function.

%%%%%%%%%%%%%%%%%%%%%%%%%%%%%%%%%%%%%%%%%%%%%%%%%%%%%%%%%%%%%%%%%%%%%%%%%%%%%%%%%%%%%%%
%%%%%%%%%%%%%%%%%%%%%%%%%%%%%%%%%%%%%%%%%%%%%%%%%%%%%%%%%%%%%%%%%%%%%%%%%%%%%%%%%%%%%%%
%%%%%%%%%%%%%%%%%%%%%%%%%%%%%%%%%%%%%%%%%%%%%%%%%%%%%%%%%%%%%%%%%%%%%%%%%%%%%%%%%%%%%%%
%%%%%%%%%%%%%%%%%%%%%%%%%%%%%%%%%%%%%%%%%%%%%%%%%%%%%%%%%%%%%%%%%%%%%%%%%%%%%%%%%%%%%%%
%%%%%%%%%%%%%%%%%%%%%%%%%%%%%%%%%%%%%%%%%%%%%%%%%%%%%%%%%%%%%%%%%%%%%%%%%%%%%%%%%%%%%%%
%%%%%%%%%%%%%%%%%%%%%%%%%%%%%%%%%%%%%%%%%%%%%%%%%%%%%%%%%%%%%%%%%%%%%%%%%%%%%%%%%%%%%%%
%%%%%%%%%%%%%%%%%%%%%%%%%%%%%%%%%%%%%%%%%%%%%%%%%%%%%%%%%%%%%%%%%%%%%%%%%%%%%%%%%%%%%%%
%%%%%%%%%%%%%%%%%%%%%%%%%%%%%%%%%%%%%%%%%%%%%%%%%%%%%%%%%%%%%%%%%%%%%%%%%%%%%%%%%%%%%%%

\section{Preliminaries}\label{sec:prelim}

\subsection{Notation}\label{sec:notation}

Here we follow \cite{Benson}. Let $G$ be a finitely generated virtually abelian group. It is well known that such a group has a normal, finite-index subgroup isomorphic to $\Z^n$ for some $n$. Throughout the paper, $G$ will denote a finitely generated virtually abelian group with normal subgroup of finite index $d$, isomorphic to $\Z^n$. Choose a finite monoid generating set $S$. If $w\in S^*$, or $W\subseteq S^*$, we will write $\overline{w}$, or $\overline{W}$, to denote the element(s) of $G$ represented by the given word(s). Write $\epsilon$ for the empty word. Let $\left|w\right|_S$ denote the length of $w$ as a word in $S^*$. In this paper, $\N$ will contain zero, and we will write $\N_+=\N\setminus\{0\}$.

A function $\omega\colon S\rightarrow\N_+$ will be called a \emph{weight function}. We extend this to $\omega\colon S^*\rightarrow\N_+$, so that $\omega(s_1s_2\cdots s_l)=\omega(s_1)+\omega(s_2)+\cdots+\omega(s_l)$ for any word $s_1s_2\cdots s_l$. Define the weight of a group element as \[\omega(g)=\mathrm{min}\{\omega(w)\mid w\in S^*,~\overline{w}=g\}.\] If $\omega(s)=1$ for all $s\in S$, this gives the usual notion of word length.

Write $X:=S\cap\Z^n=\{x_1,\ldots,x_r\}$ and $Y:=S\setminus(S\cap\Z^n)=\{y_1,\ldots,y_s\}$, and call any word in $Y^*$ a \emph{pattern}.

\begin{definition}
Let $\p\colon S\rightarrow Y$ be the map \[\p\colon s_i\mapsto\begin{cases} \epsilon & \text{ if }s_i\in X\\ s_i & \text{ if }s_i\in Y \end{cases}.\] This extends to a monoid homomorphism $\p\colon S^*\rightarrow Y^*$, which records those generators in a word which are not contained in $\Z^n$. We call $\p(w)$ the \emph{pattern} of $w$.
\end{definition}

\begin{definition}
For a pattern $\pi=y_{i_1}y_{i_2}\cdots y_{i_k}$ of length $k$, define the set of $\pi$-patterned words in $S^*$ as
\begin{equation*}
W^\pi =\{x_1^{w_1}\cdots x_r^{w_r}y_{i_1}x_1^{w_{r+1}}\cdots x_r^{w_{2r}}y_{i_2}x_1^{w_{2r+1}}\cdots y_{i_k}x_1^{w_{kr+1}}\cdots x_r^{w_{kr+r}}\mid w_j\in\N\}.
\end{equation*}
\end{definition}
Any group element represented by a word with pattern $\pi$ can be represented by a word in $W^\pi$ (since powers of elements of $X$ commute). For a fixed pattern $\pi$, $W^\pi$ is in one-to-one correspondence with non-negative integer vectors of length $kr+r$. We define $m(\pi)=kr+r$. When it is clear which pattern this refers to, we will just write $m$.
\begin{definition}
Define a map $\varphi\colon W^\pi\rightarrow\N^m$ via 
\begin{align*}
x_1^{w_1}\cdots x_r^{w_r}y_{i_1}x_1^{w_{r+1}}\cdots x_r^{w_{2r}}y_{i_2} x_1^{w_{2r+1}}\cdots y_{i_k}x_1^{w_{kr+1}}\cdots x_r^{w_{kr+r}} \mapsto \begin{pmatrix}w_1 \\ w_2 \\ \vdots \\ w_m\end{pmatrix},
\end{align*}
which records the powers of the generators contained in $\Z^n$. For $w\in W^\pi$, write $\w:=\varphi(w)$, and for a subset $V\subseteq W^\pi$, write
\begin{equation*}
\overrightarrow{V}:=\varphi(V)\in\N^m.
\end{equation*}
Note that $\varphi$ is a bijection.
\end{definition}

Let $w\in W^\pi$ with the above form. Then the weight of $w$ is given by 
\begin{equation}\label{eq:weightsum}
\omega(w)=\sum_{j=0}^k\sum_{i=1}^r\omega(x_i)w_{jr+i}+\sum_{j=1}^k\omega(y_{i_j}).
\end{equation}

Let $T\subset G$ be a choice of coset representatives for $\Z^n\backslash G$. We can then express an element of $G$ uniquely as $(a_1,\ldots,a_n)^Tt$ for $a_j\in\Z$, $t\in T$. In order to pass from a patterned word to this standard form, we introduce some constants.
\begin{notation}
For each $x_i\in X$, let $x_i=(z_{1i},z_{2i},\ldots,z_{ni})^T\in\Z^n$, and define the $n\times r$ matrix \[Z:=\left[z_{ji}\right]_{1\leq j\leq n,1\leq i\leq r}=\begin{pmatrix} \vrule & \vrule & & \vrule \\ x_1 & x_2 & \cdots & x_n \\ \vrule & \vrule & & \vrule \end{pmatrix},\] which encodes the abelian part of the generating set.

We then have
\begin{equation}\label{eq:wordidentity1}
x_1^{w_1}x_2^{w_2}\cdots x_r^{w_r}= Z\begin{pmatrix}w_1\\ \vdots \\ w_r\end{pmatrix}\in\Z^n
\end{equation}
for any powers $w_i\in\N$.

Now we encode conjugation of elements of $\Z^n$ via matrix multiplication. Let $e_i\in\Z^n$ be the $i$th standard basis vector, and $y_k\in Y$. Then $y_ke_iy_k^{-1}\in\Z^n\lhd G$, and we will let $y_ke_iy_k^{-1}=(\gamma_{1i,k},\gamma_{2i,k},\ldots,\gamma_{ni,k})$ for each $1\leq i\leq n$, $y_k\in Y$. Define the $n\times n$ matrix \[\Gamma_k:=\left[\gamma_{ji,k}\right]_{1\leq i,j\leq n}=\begin{pmatrix} \horz y_ke_1y_k^{-1} \horz\\ \horz y_ke_2y_k^{-1} \horz\\ \vdots \\ \horz y_ke_ny_k^{-1} \horz \end{pmatrix}^T. \]

Then
\begin{align}
y_k\begin{pmatrix} a_1 \\ a_2 \\ \vdots \\ a_n\end{pmatrix}y_k^{-1} = \Gamma_k\begin{pmatrix} a_1 \\ a_2 \\ \vdots \\ a_n\end{pmatrix} \label{eq:wordidentity2}
\end{align}
for any $y_k\in Y$, $a_i\in\Z$.

So we can express powers of the $x_i$ generators as vectors in $\Z^n$, and we can move powers of $y_k$ past such vectors by using the identity $y_k(a_1,\ldots,a_n)=\Gamma_k(a_1,\ldots,a_n)^Ty_k$. For the word \[w=x_1^{w_1}x_2^{w_2}\cdots x_r^{w_r}y_{i_1}x_1^{w_{r+1}}\cdots x_r^{w_{2r}}y_{i_2}x_1^{w_{2r+1}}\cdots y_{i_k}x_1^{w_{kr+1}}\cdots x_r^{w_{kr+r}},\] we can use the identities \eqref{eq:wordidentity1} and \eqref{eq:wordidentity2} to move all the $y_k$ generators to the right, modifying the powers of $x_i$ as we go, without changing the element that is represented. Thus
\begin{equation}\label{eq:wbargamma}
\overline{w}=\left\lbrace Z\begin{pmatrix}w_1 \\ \vdots \\w_r\end{pmatrix}+\Gamma_{i_1}Z\begin{pmatrix}w_{r+1} \\ \vdots \\w_{2r}\end{pmatrix}+\cdots+\Gamma_{i_1}\Gamma_{i_2}\cdots\Gamma_{i_k}Z\begin{pmatrix}w_{kr+1} \\ \vdots \\ w_m\end{pmatrix}\right\rbrace\overline{\pi}
\end{equation}

\end{notation}

To express this more compactly, we introduce further notation.
\begin{definition}\label{def:Ans}
Consider the $n\times m$ matrix formed by placing the matrices $Z$, $\Gamma_{i_1}Z$, $\Gamma_{i_1}\Gamma_{i_2}Z\ldots$ next to each other. The transposes of the rows of this new matrix are $m$-dimensional vectors which we will call $A_i^\pi$ as follows:
\begin{equation*}
\begin{pmatrix} \horz  \left(A_1^{\pi}\right)^T  \horz \\ \horz  \left(A_2^{\pi}\right)^T  \horz \\ \vdots \\ \horz \left(A_n^{\pi}\right)^T \horz \end{pmatrix} :=
\left(\begin{array}{c|c|c|c|c} Z & \Gamma_{i_1}Z & \Gamma_{i_1}\Gamma_{i_2}Z & \cdots & \Gamma_{i_1}\Gamma_{i_2}\cdots\Gamma_{i_k}Z\end{array}\right).
\end{equation*}
\end{definition}
\begin{definition}\label{def:Bns}
The word $\pi$ itself represents some element of $G$, so we introduce integers $B_i^\pi$ so that we can write $\overline{\pi}$ in the standard form given by the choice of coset representatives as \[\overline{\pi}=\begin{pmatrix} B_1^\pi \\ B_2^\pi \\ \vdots \\ B_n^\pi\end{pmatrix}t_\pi,\] for $t_\pi\in T$.
\end{definition}

Now we may rewrite equation~\eqref{eq:wbargamma} using scalar products as
\begin{equation}\label{eq:wbar}
\overline{w}=\left\lbrace \begin{pmatrix}A_1^\pi\cdot \w \\ A_2^\pi\cdot \w \\ \vdots \\ A_n^\pi\cdot \w\end{pmatrix}+ \begin{pmatrix} B_1^\pi \\ B_2^\pi \\ \vdots \\ B_n^\pi\end{pmatrix}\right\rbrace t_\pi.
\end{equation}
So words in $W^\pi$ represent the same element of $G$ if and only if the scalar products of the corresponding vectors with the $A_i^\pi$s agree.
\begin{remark}\label{rem:patternsandcosets}
	We emphasise that $A_i^\pi\in\Z^m$, $B_i^\pi\in\Z$, and $t_\pi\in G$ are constant in the sense that they depend only on the pattern $\pi$. In particular, any two words with the same pattern represent elements of the same coset.
\end{remark}

\begin{definition}\label{An+1}
	For a pattern $\pi=y_{i_1}\cdots y_{i_k}$ of length $k$, let $A_{n+1}^\pi\in\N^m$ record the weights of the $x_i$ generators, ordered as follows: 
\begin{equation*}
A_{n+1}^\pi=\underbrace{\left(\omega(x_1),\omega(x_2),\ldots,\omega(x_r),\ldots,\omega(x_1),\omega(x_2),\ldots,\omega(x_r)\right)^T}_{r\text{ weights, repeated }k+1\text{ times, giving an }m\text{-dimensional vector}}.
\end{equation*}
Furthermore, let $B_{n+1}^\pi$ record the weight of the word $\pi$, i.e. \[B_{n+1}^\pi=\sum_{j=1}^k\omega(y_{i_j})\]
\end{definition}
We can then express equation \eqref{eq:weightsum} more compactly using a scalar product: \begin{equation}\label{eq:weight}
\omega(w)=A_{n+1}^\pi\cdot \w + B_{n+1}^\pi.
\end{equation}

We now have a collection of vectors $A_i^\pi$ and integers $B_i^\pi$ which together put words in $W^\pi$ into the chosen standard form.

\subsection{Structure constants for testing conjugacy}\label{sec:conjconstants}
In what follows we introduce notation that we will need to prove Theorem \ref{thm:conj}. As above, we fix a transversal $T$ for $\Z^n\backslash G$.

In a similar manner to above, we encode conjugation of an element \[\begin{pmatrix}a_1\\\vdots\\a_n\end{pmatrix}\in\Z^n\lhd G\] by some other element of $G$ via multiplication by a matrix. Let $t\in T$, and $e_i$ be the $i$th standard basis vector in $\Z^n$. Then $te_it^{-1}\in\Z^n$ by normality. Let $te_it^{-1}=(\delta_{1i,t},\delta_{2i,t},\ldots,\delta_{ni,t})^T$, and write $\Delta_t$ for the $n\times n$ matrix $[\delta_{ji,t}]_{1\leq i,j\leq n}$ whose columns are $te_it^{-1}$. Then we have \[t\begin{pmatrix}a_1\\\vdots\\a_n\end{pmatrix}t^{-1}=\Delta_t\begin{pmatrix}a_1\\\vdots\\a_n\end{pmatrix}\] for any $a_i\in\Z$.

Fix a pattern $\pi$. Recall the matrices $Z$, and $\Gamma_j$ for each element $y_j\in Y$. If \[w=x_1^{w_1}x_2^{w_2}\cdots x_r^{w_r}y_{j_1}x_1^{w_{r+1}}x_2^{w_{r+2}}\cdots x_r^{w_{2r}}y_{i_2}\cdots y_{i_k}x_1^{w_{kr+1}}\cdots x_r^{w_m}\in W^\pi,\]
consider the group element $t\overline{w}t^{-1}$ for $t\in T$. With $\Delta_t$ as defined above, equation \eqref{eq:wbargamma} yields
\begin{align}
t\overline{w}t^{-1}=  \left\lbrace \Delta_tZ\begin{pmatrix}w_1 \\ \vdots \\w_r \end{pmatrix}+\right.&\left.\Delta_t\Gamma_{j_1}Z\begin{pmatrix}w_{r+1}\\ \vdots \\w_{2r}\end{pmatrix}+\cdots\right. \nonumber\\
&\left.+~\Delta_t\Gamma_{j_1}\cdots\Gamma_{j_k}Z\begin{pmatrix} w_{kr+1}\\ \vdots \\ w_{m}\end{pmatrix}\right\rbrace t\overline{\pi}t^{-1}. \label{tconj}
\end{align}
We make definitions analogous to \ref{def:Ans} and \ref{def:Bns} above.
\begin{definition}
For each pattern $\pi$, and $t\in T$, place the matrix products in order and define $m$-dimensional vectors $A_{i,t}^\pi$ as follows
\begin{equation*}
\begin{pmatrix} \horz  \left(A_{1,t}^{\pi}\right)^T  \horz \\ \horz  \left(A_{2,t}^{\pi}\right)^T  \horz \\ \vdots \\ \horz \left(A_{n,t}^{\pi}\right)^T \horz  \end{pmatrix} :=
\left(\begin{array}{c|c|c|c} \Delta_t Z & \Delta_t\Gamma_{j_1}Z & \cdots & \Delta_t\Gamma_{j_1}\ldots\Gamma_{j_k}Z\end{array}\right).
\end{equation*}
\end{definition}
Now we can express equation \eqref{tconj} in terms of scalar products:
\begin{equation*}
t\overline{w}t^{-1} =  \begin{pmatrix} A_{1,t}^\pi\cdot \w \\ A_{2,t}^\pi\cdot\w \\ \vdots \\ A_{n,t}^\pi\cdot \w \end{pmatrix} t\overline{\pi}t^{-1}.
\end{equation*}
\begin{definition}
We have $t\overline{\pi}t^{-1}=xs$ for some $s\in T$, $x\in\Z^n$. This $x$ depends only on $\pi$ and $t$. Write $x=\left(B_{1,t}^\pi,\ldots,B_{n,t}^\pi\right)^T$.
\end{definition}
Thus we have
\begin{equation}\label{eq:wbarconj}
t\overline{w}t^{-1} = \left\lbrace\begin{pmatrix} A_{1,t}^\pi\cdot \w \\ A_{2,t}^\pi\cdot\w \\ \vdots \\ A_{n,t}^\pi\cdot \w \end{pmatrix}+\begin{pmatrix} B_{1,t}^\pi \\ B_{2,t}^\pi \\ \vdots \\ B_{n,t}^\pi\end{pmatrix}\right\rbrace s.
\end{equation}
Now we have a collection of vectors that we can use to test whether two words represent conjugate elements (provided we know that their coset representatives are conjugate). The following Lemma makes this precise.

\begin{lemma}\label{lem:detectconj}
	Let $v\in W^\pi$ and $w\in W^\mu$ for patterns $\pi$ and $\mu$, and let $t\in T$. Then $\overline{v}=t\overline{w}t^{-1}$ if and only if
	\begin{enumerate}
		\item $\overline{\pi}$ and $t\overline{\mu}t^{-1}$ are in the same $\Z^n$-coset, and
		\item $A_i^\pi\cdot \v+B_i^\pi = A_{i,t}^\mu\cdot \w + B_{i,t}^\mu$ for each $1\leq i\leq n$.
	\end{enumerate}
\end{lemma}
\begin{proof}
	From equation \eqref{eq:wbar} we have \[\overline{v}=\left\{\left(A_1^\pi\cdot \w,A_2^\pi\cdot \w,\ldots,A_n^\pi\cdot \w\right)^T+\left(B_1^\pi,B_2^\pi,\ldots,B_n^\pi\right)^T\right\}t_\pi.\] From equation \eqref{eq:wbarconj} we have \[t\overline{w}t^{-1} = \left\lbrace\left(A_{1,t}^\mu\cdot \w,A_{2,t}^\mu\cdot \w,\ldots,A_{n,t}^\mu\cdot \w\right)^T+\left(B_{1,t}^\mu,B_{2,t}^\mu,\ldots,B_{n,t}^\mu\right)^T\right\rbrace s\] where $s$ is determined by the element $t\overline{\mu}t^{-1}$, and the result follows.
\end{proof}

\subsection{Reduction to finitely many patterns}\label{sec:finitepatterns}
In general, there are infinitely many patterns in $Y^*$. It will be essential for the results on growth series that we have only finitely many patterns. Extend $S$ to a new generating set
\begin{equation}\label{eq:Stilde}
\widetilde{S}=\{\overline{s_1s_2\cdots s_k}\mid s_i\in S,~1\leq k\leq d\}.
\end{equation}

For each $\tilde{s}\in\widetilde{S}$, assign the weight $\omega(\tilde{s})$, which induces a weighted length function $\wt{\omega}$ on the words in $\wt{S}^*$, and in turn on elements of $G$. This function preserves the weight of elements of $G$ (although the length will change), meaning the growth functions are equal:
\begin{equation*}
 \sum_{n\geq0}\sigma_{G,S}^\omega(n)z^n=\sum_{n\geq0}\sigma_{G,\wt{S}}^{\wt{\omega}}(n)z^n.
\end{equation*}

\begin{proposition}[Proposition 11.3 of \cite{Benson}]\label{elementary}
 Suppose $g\in G$. Then there exists a minimal weight word $w\in\wt{S}^*$, with a pattern of length at most $d$ (with respect to $\wt{S}^*$), such that $\overline{w}=g$.
\end{proposition}

This leads us to make the following definition.
\begin{definition}\label{def:finitepatterns}
	For a virtually abelian group $G$ generated by $S$, let \[P:=\{\pi\in(\wt{S}\setminus(\wt{S}\cap\Z^n))^*\mid\left|\pi\right|_{\wt{S}}\leq d\}.\] Note that $P$ is a finite set.
\end{definition}
Then the previous proposition implies that any $g\in G$ can be represented by a minimal weight word with a pattern in $P$.

\subsection{Polyhedral Sets}\label{sec:polyhedral}
The following definition and results follow Benson's work \cite{Benson}. However, the ideas appear in model theory as early as Presburger \cite{Presburger}. Results regarding rationality can be found in \cite{Denef}, and the ideas appear in the theory of Igusa local zeta functions (see \cite{Clucketal}). These last are linked to subgroup growth \cite{GSS}, a different notion of growth in groups not considered here.

\begin{definition}
Let $m\in\N_+$, and let $\cdot$ denote the Euclidean scalar product. Then for any $\boldsymbol{u}\in\Z^m$, $a\in\Z$, $b\in\N_+$:
\begin{enumerate}
\item an \emph{elementary set} is any subset of $\Z^m$ of the form $\{\boldsymbol{z}\in\Z^m\mid\boldsymbol{u}\cdot\boldsymbol{z}=a\}$, $\{\boldsymbol{z}\in\Z^m\mid\boldsymbol{u}\cdot\boldsymbol{z}>a\}$, or $\{\boldsymbol{z}\in\Z^m\mid\boldsymbol{u}\cdot\boldsymbol{z}\equiv a\mod b\}$,
\item a \emph{basic polyhedral set} is any finite intersection of elementary sets,
\item a \emph{polyhedral set} is any finite union of basic polyhedral sets.
\end{enumerate}
If $\cP\subset\Z^m$ is polyhedral and additionally $\cP\subseteq\N^m$, we call $\cP$ a \emph{positive polyhedral set}.
\end{definition}
We record some crucial facts about polyhedral sets.
\begin{proposition}[Proposition 13.1 of \cite{Benson}]\label{prop:polyclosed}
	Polyhedral sets in $\Z^m$ are closed under finite unions, finite intersections, and set complement.
\end{proposition}
\begin{proposition}[Propositions 13.7 and 13.8 of \cite{Benson}]\label{prop:polyaffine}
	Let $\cE\colon\Z^m\rightarrow\Z^{m'}$ be an integral affine transformation (for some $m,m'>0$). That is, there is some $m'\times m$ matrix with integer entries and some $q\in\Z^{m'}$ such that $\cE(p)=Ap+q$ for $p\in\Z^m$. If $\cP\subseteq\Z^m$ is a polyhedral set then $\cE(\cP)\subseteq\Z^{m'}$ is a polyhedral set. If $\cQ\subseteq\Z^{m'}$ is a polyhedral set then the preimage $\cE^{-1}(\cQ)\subseteq\Z^m$ is a polyhedral set.
\end{proposition}
We note that projection onto any subset of the coordinates of $\Z^m$ is an integral affine transformation.

Let $\cP\subseteq\N^m$ be a positive polyhedral set. Given some choice of weights $(\omega_1,\ldots,\omega_m)$ for the coordinates of $\N^m$, we assign the weight $\sum_{i=1}^m a_i\omega_i$ to the vector $(a_1,\ldots,a_m)^T\in \cP$. Define \[\sigma^\omega_\cP(n)=\#\{p\in\cP\mid\omega(p)=n\},\] and the resulting weighted growth series \[\cS^\omega_\cP(z)=\sum_{n=0}^\infty\sigma^\omega_\cP(n)z^n.\]
\begin{proposition}[Proposition 14.1 of \cite{Benson}]\label{prop:polyhedralrationalgrowth}
	 If $\cP$ is a positive polyhedral set, the weighted growth series $\cS^\omega_\cP(z)$ is a rational function.
\end{proposition}

Benson's result that virtually abelian groups have rational weighted growth series with respect to all generating sets follows from the following Theorem.
\begin{theorem}[Theorem 1.2 of \cite{Benson}]\label{thm:Upi}
	Let $G$ be a virtually abelian group, with any choice of finite weighted generating set $S$. For each pattern $\pi\in P$, with $P$ as in definition \ref{def:finitepatterns}, there exists a set $U^\pi\subset\wt{S}^*$ such that $\overrightarrow{U^\pi}\subset\N^{m(\pi)}$ is polyhedral, and the disjoint union $\bigcup_{\pi\in P}U^\pi$ forms a language of unique minimal-weight representatives for the elements of $G$.
\end{theorem}
Since the $x_i$ generators in $U^\pi$ correspond to the coordinates of $\N^{m(\pi)}$, and the contribution from the $y_k$ generators is constant within $U^\pi$, rational growth of $\overrightarrow{U^\pi}$ (in the sense of Proposition \ref{prop:polyhedralrationalgrowth}) implies rational growth of $U^\pi$ with respect to $\widetilde{S}$, and hence with respect to $S$.

We will need the following Lemma concerning polyhedral sets.
\begin{lemma}\label{lem:polyhedralcover}
	Let $\cP$ be a polyhedral subset of $\Z^m$ for some $m\geq1$. Suppose there exist polyhedral sets $X_1,\ldots,X_k\subset\Z^m$ such that $\cP\subseteq\bigcup_{i=1}^k X_i$. Then there exist polyhedral sets $Y_i\subseteq X_i$ for each $i$ such that $Y_i\cap Y_j=\emptyset$ for $i\neq j$ and $\cP=\bigcup_{i=1}^k Y_i$.
\end{lemma}
\begin{proof}
	We induct on $k$. Consider $k=1$, i.e. $\cP\subseteq X_1$. Let $Y_1=\cP\cap X_1$, which is polyhedral as an intersection of polyhedral sets. In other words $\cP=Y_1$.
	
	Now assume the statement is true for some $k>1$. Let $\cP$ be some polyhedral set, with polyhedral sets $X_1,\ldots,X_{k+1}$ such that $\cP\subseteq\bigcup_{i=1}^{k+1}X_i$. Consider the polyhedral set \[\cQ:=\cP\cap\bigcup_{i=1}^k X_i.\] Now since $\cQ\subseteq\bigcup_{i=1}^k X_i$, the inductive hypothesis gives polyhedral sets $Y_i\subseteq X_i$ for $1\leq i\leq k$ such that $Y_i\cap Y_j=\emptyset$ for each $i\neq j$, $1\leq i,j\leq k$, and \[\cQ=\bigcup_{i=1}^k Y_i.\] Now let $Y_{k+1}=\cP\setminus\cQ$, also a polyhedral set. We have $Y_{k+1}\cap\cQ=\emptyset$, and $Y_{k+1}\subseteq X_{k+1}$ (since the $X_i$s cover $\cP$ and $Y_{k+1}$ does not intersect $\cQ$), and by definition \[\cP=\cQ\cup Y_{k+1}=\bigcup_{i=1}^{k+1}Y_i.\] So the statement holds for $k+1$.
\end{proof}

\subsection{$N$-fold patterns}\label{sec:N-foldpatterns}
We develop a framework for dealing with $N$-tuples of patterned words for some finite $N$.
\begin{definition}
Let $Q$ be any set of patterns (for a virtually abelian group with some choice of finite generating set). Let $N$ be a positive integer. An \emph{$N$-fold pattern} will be an $N$-tuple of patterns from $Q$. We will write $\bs{\pi}=(\pi_1,\pi_2,\ldots,\pi_N)\in Q^N$.

Given an $N$-fold pattern $\bs{\pi}=(\pi_1,\pi_2,\ldots,\pi_N)$, write $W^{\bs{\pi}}$ for the set of $N$-tuples of words with patterns given by $\bs{\pi}$. More precisely,
\begin{equation*}
W^{\bs{\pi}}=\left\{\left(w^{(1)},w^{(2)},\ldots,w^{(N)}\right)\,\middle\vert\, w^{(i)}\in W^{\pi_i},~1\leq i\leq N \right\}
\end{equation*}
Let $m(\bs{\pi})=\sum_i m(\pi_i)$. As above, we extract the powers of the $x_i$ generators with respect to each pattern and note that elements of $W^{\bs{\pi}}$ are in one-to-one correspondence with vectors in $\N^{m(\bs{\pi})}$ via \[\left(w^{(1)},\ldots,w^{(N)}\right)\mapsto\begin{pmatrix}\overrightarrow{w^{(1)}} \\ \vdots \\ \overrightarrow{w^{(N)}}\end{pmatrix}.\]
\end{definition}

In the following Lemma, we show that if we have a polyhedral set of $N$-tuples in $\N^{m(\bs{\pi})}$, we can extract a minimal-weight element of each tuple, in a manner that preserves rational growth.
\begin{lemma}\label{lem:minimaltuplerep}
	Let $G$ be virtually abelian, generated by $S$, with weight function $\omega$. Fix some $N$-fold pattern $\bs{\pi}=(\pi_1,\pi_2,\ldots,\pi_N)$. Let $V\subset W^{\bs{\pi}}$. If $\overrightarrow{V}$ is a polyhedral set then there exists a language $\cL\subset S^*$ in one-to-one correspondence with $V$ such that for each $(v_1,\ldots,v_N)\in V$, $\cL$ contains exactly one element $w\in\{v_1,\ldots,v_N\}$, with $\omega(w)\leq\omega(v_i)$ for each $1\leq i\leq N$, and such that $\cL$ has rational weighted growth series.
\end{lemma}

\begin{proof}
	Let \[X_j(\bs{\pi})=\left\{\left(w^{(1)},\ldots,w^{(N)}\right)\in W^{\bs{\pi}}\,\middle\vert\,\omega\left(w^{(j)}\right)\leq\omega\left(w^{(k)}\right),~1\leq k\leq N\right\}\] be those elements in $W^{\bs{\pi}}$ where the $j$th component word has minimal weight. For each $\pi_k$, define the $m(\bs{\pi})$-dimensional vector $D^{\pi_k}$ as follows:
\begin{equation*}
D^{\pi_k}:=\begin{pmatrix} 0 \\ \vdots \\ 0 \\ A_{n+1}^{\pi_k} \\ 0 \\ \vdots \\ 0 \end{pmatrix}
\begin{tabular}{l}
$\left.\lefteqn{\phantom{\begin{matrix} 0 \\ \vdots \\ 0 \end{matrix}}}\right\}$ \small{$\sum_{i=1}^{k-1}m(\pi_i)$ zeros}\\
$\lefteqn{\phantom{\begin{matrix} A_{n+1}^{\pi_j} \end{matrix}}}$ \small{$m(\pi_k)$ rows} \\
$\left.\lefteqn{\phantom{\begin{matrix} 0 \\ \vdots \\ 0 \end{matrix}}}\right\}$ \small{$\sum_{i=k+1}^{N}m(\pi_i)$ zeros} \\
\end{tabular}
\end{equation*}
Then for some $N$-tuple of words $\left(w^{(1)},\ldots,w^{(N)}\right)\in W^{\bs{\pi}}$, whose corresponding vector is $\vec{z}\in\N^{m(\bs{\pi})}$, we have $D^{\pi_k}\cdot\vec{z}=A_{n+1}^{\pi_k}\cdot\overrightarrow{w^{(k)}}$.
	
	For $w\in W^{\pi_j}$ and $w'\in W^{\pi_k}$, equation \eqref{eq:weight} implies that \[\omega(w)\leq\omega(w')\Leftrightarrow A_{n+1}^{\pi_j}\cdot \w+B_{n+1}^{\pi_j}\leq A_{n+1}^{\pi_k}\cdot \vec{w'}+B_{n+1}^{\pi_k}.\] Therefore we have 
	\begin{align*}\overrightarrow{X_j(\bs{\pi})}&=\bigcap_{\substack{1\leq k\leq N \\ k\neq j}} \left\lbrace \vec{z}\in\N^{m(\bs{\pi})} \mid D^{\pi_j}\cdot \vec{z} +B_{n+1}^{\pi_j}  \leq D^{\pi_k}\cdot \vec{z} +B_{n+1}^{\pi_k}\right\rbrace \\
	&=\bigcap_{\substack{1\leq k\leq N \\ k\neq j}} \left\lbrace \vec{z}\in\N^{m(\bs{\pi})} \mid \left(D^{\pi_j}-D^{\pi_k}\right)\cdot \vec{z}   \leq B_{n+1}^{\pi_k} - B_{n+1}^{\pi_j}\right\rbrace.
	\end{align*}
	Thus each $\overrightarrow{X_j(\bs{\pi})}$ is a polyhedral subset of $\N^{m(\bs{\pi})}$ (since it is a finite intersection of elementary sets). Since every $N$-tuple in $V$ must have at least one minimal-weight element, we have $\overrightarrow{V}\subseteq\bigcup_{j=1}^N \overrightarrow{X_j(\bs{\pi})}$. Lemma \ref{lem:polyhedralcover} then gives us polyhedral sets $\overrightarrow{Y_j(\bs{\pi})}\subseteq \overrightarrow{X_j(\bs{\pi})}$ for each $j$ which are pairwise disjoint, and such that $\overrightarrow{V}=\bigcup_{j=1}^N \overrightarrow{Y_j(\bs{\pi})}$.
	
	Let $\rho_j\colon\N^{m(\bs{\pi})}\rightarrow\N^{m(\pi_j)}$ be the projection onto just the coordinates corresponding to the $j$th component of the $N$-tuple. Since images of polyhedral sets under projections are still polyhedral by Proposition \ref{prop:polyaffine}, we have $N$ polyhedral sets of the form $\rho_j(\overrightarrow{Y_j(\bs{\pi})})$, each corresponding to a collection of words which are minimal-weight representatives for their $N$-tuple, and each growing rationally with respect to the generators of $G$ (by Proposition \ref{prop:polyhedralrationalgrowth}). Since the $\overrightarrow{Y_j}$s are disjoint, and cover $\overrightarrow{V}$, the union of these polyhedral sets corresponds to a language of unique, minimal-weight representatives for $V$, which grows rationally.
\end{proof}

%%%%%%%%%%%%%%%%%%%%%%%%%%%%%%%%%%%%%%%%%%%%%%%%%%%%%%%%%%%%%%%%%%%%%%%%%%%%%%%%%%%%%%%%%
%%%%%%%%%%%%%%%%%%%%%%%%%%%%%%%%%%%%%%%%%%%%%%%%%%%%%%%%%%%%%%%%%%%%%%%%%%%%%%%%%%%%%%%%5
%%%%%%%%%%%%%%%%%%%%%%%%%%%%%%%%%%%%%%%%%%%%%%%%%%%%%%%%%%%%%%%%%%%%%%%%%%%%%%%%%%%%%%%%%
%%%%%%%%%%%%%%%%%%%%%%%%%%%%%%%%%%%%%%%%%%%%%%%%%%%%%%%%%%%%%%%%%%%%%%%%%%%%%%%%%%%%%%%%%
%%%%%%%%%%%%%%%%%%%%%%%%%%%%%%%%%%%%%%%%%%%%%%%%%%%%%%%%%%%%%%%%%%%%%%%%%%%%%%%%%%%%%%%%55
%%%%%%%%%%%%%%%%%%%%%%%%%%%%%%%%%%%%%%%%%%%%%%%%%%%%%%%%%%%%%%%%%%%%%%%%%%%%%%%%%%%%%%%%%%%

\section{Relative growth series}\label{sec:Relative}

Given a group $G$ with finite generating set $S$, consider a subset $H\subset G$. One can construct the relative (weighted) growth series of $H$ with respect to the generators of $G$ by defining $\sigma(n)$ to be the number of elements of $H$ with weight $n$. We prove that for any \emph{subgroup} of a virtually abelian group, the relative weighted growth series is rational (for any finite generating set and positive integer weight function).

We do this by showing that we can use linear algebra to test whether or not a word in $S^*$ represents an element of $H$, and thus describe the language of words whose image in $G$ is contained in $H$ in terms of polyhedral sets. We then intersect with the language of minimal representatives for the elements of $G$ to obtain a language of minimal representatives for the elements of $H$, which grows rationally.

The following two Lemmas are elementary and are stated without proof.
\begin{lemma}\label{lem:vasubgroup}
	Any subgroup of a virtually abelian group is virtually abelian.
\end{lemma}
In the next Lemma, we establish criteria for an element of a virtually abelian group to lie in a chosen subgroup.
\begin{lemma}\label{lem:subgroupcriteria}
	Let $G$ be a finitely generated virtually abelian group, with normal subgroup $\Z^n$ of finite index $d$ as usual. Let $H<G$, and choose a set $\{h_1,\ldots,h_c\}\in H$ of coset representatives for $H\cap\Z^n\backslash H$. Then this set can be extended to a set $\{h_1,\ldots,h_c,\ldots,h_d\}$ of coset representatives for $\Z^n\backslash G$, and an element $xh_i\in G$, with $x\in\Z^n$, is in $H$ if and only if 
	\begin{enumerate}
	\item $1\leq i\leq c$, and\label{subgroupcrit1}
	\item $x\in H\cap\Z^n$.\label{subgroupcrit2}
	\end{enumerate}
\end{lemma}

\begin{theorem}\label{thm:relative}
	Let $G$ be a finitely generated virtually abelian group and let $H$ be any subgroup of $G$. Then $H$ has rational weighted growth series relative to any choice of generators of $G$ (with any weight function).
\end{theorem}
\begin{proof}
	The strategy of the proof is to use the criteria given in Lemma \ref{lem:subgroupcriteria} to detect those elements of $G$ which are contained in $H$, pattern by pattern, and show that they form polyhedral sets. We can then use the sets $U^\pi$ (from Theorem \ref{thm:Upi}) to find minimal representatives for the elements of $H$.
	
	Fix a generating set $S$ for $G$, and a weight function $\omega$. We consider the expanded generating set $\wt{S}$ of Section \ref{sec:finitepatterns}, and the corresponding finite set of patterns $P$ as in Definition \ref{def:finitepatterns}. Since words with the same pattern represent elements of the same $\Z^n$-coset (see Remark \ref{rem:patternsandcosets}), only words whose pattern represents an element of $\bigsqcup_{i=1}^c\Z^nh_i$ can possibly be in $H$. Call the set of such patterns $P_H$.
	
	Fix one of these patterns, $\pi\in P_H$, and the resulting vectors $A^\pi_i$ and integers $B^\pi_i$ as in section \ref{sec:prelim}. Consider a word $w\in W^\pi$. By design, we have an $h_i$ with $1\leq i\leq c$ so that $\overline{w}\in\Z^n h_i$, so criterion \eqref{subgroupcrit1} is satisfied. Now by criterion \eqref{subgroupcrit2}, $w$ represents an element of $H$ if and only if $\overline{w}\in (H\cap\Z^n)h_i\subset\Z^nh_i$. Since \[\overline{w}=\left\lbrace\begin{pmatrix}A_1^\pi\cdot\w \\ A_2^\pi\cdot\w \\ \vdots \\ A_n^\pi\cdot\w\end{pmatrix}+\begin{pmatrix}B_1^\pi \\ B_2^\pi \\ \vdots \\ B_n^\pi\end{pmatrix}\right\rbrace h_i,\] $\overline{w}\in H$ if and only if
	\begin{equation}\label{dotprod3}
	\begin{pmatrix}A_1^\pi\cdot\w \\ A_2^\pi\cdot\w \\ \vdots \\ A_n^\pi\cdot\w\end{pmatrix}+\begin{pmatrix}B_1^\pi \\ B_2^\pi \\ \vdots \\ B_n^\pi\end{pmatrix}\in H\cap\Z^n.
	\end{equation}
	
	Now $H\cap\Z^n$ is a (free abelian) subgroup of $\Z^n$. Suppose it has dimension $f$, and choose a basis $\{\bs{b}_1,\ldots,\bs{b}_f\}\subset\Z^n$. Then $w$ satisfies \eqref{dotprod3} if and only if there exist integers $a_1,\ldots,a_f$ so that \[\begin{pmatrix}A_1^\pi\cdot\w \\ A_2^\pi\cdot\w \\ \vdots \\ A_n^\pi\cdot\w\end{pmatrix}+\begin{pmatrix}B_1^\pi \\ B_2^\pi \\ \vdots \\ B_n^\pi\end{pmatrix}=a_1\bs{b}_1+\cdots+a_f\bs{b}_f.\] In other words, 
	\begin{equation}\label{Hcriterion}
	A_i^\pi\cdot\w+B_i^\pi=e_i\cdot(a_1\bs{b}_1+\cdots+a_f\bs{b}_f)=a_1(e_i\cdot\bs{b}_1)+\cdots+a_f(e_1\cdot\bs{b}_f)
	\end{equation}
	for each $1\leq i\leq n$. 
	
	We will now express the set of all vectors satisfying \eqref{Hcriterion} as a polyhedral set. For each $1\leq i\leq n$, define the $(m+f)$-dimensional vector 
	\begin{equation*}
	C_i^\pi=\begin{pmatrix} \vrule \\ A_i^\pi \\ \vrule \\ -e_i\cdot \bs{b}_1 \\ \vdots \\ -e_i\cdot \bs{b}_f \end{pmatrix}
	\begin{tabular}{l}
	$\left.\lefteqn{\phantom{\begin{pmatrix} \vrule \\ A_i^\pi \\ \vrule \end{pmatrix}}}\right\}$ $m$ rows \\
	$\left.\lefteqn{\phantom{\begin{pmatrix} -e_i\cdot \bs{b}_1 \\ \vdots \\ -e_i\cdot \bs{b}_f \end{pmatrix}}}\right\}$ $f$ rows
	\end{tabular}
	\end{equation*}
	 where the first $m$ entries are the entries of the vector $A_i^\pi$, and the last $f$ entries are $-1$ times the $i$th component of the basis vectors. Now a vector $\w\in\Z^m$ satisfies \eqref{Hcriterion} for some $i$ precisely when there is a vector $\v\in\Z^{m+f}$, with entries $v_1,v_2,\ldots,v_{m+f}$, such that $\w=(v_1,\ldots,v_m)^T$ and $A_i^\pi\cdot\w + B_i^\pi=e_i\cdot(v_{m+1}\bs{b}_1+\cdots+v_{m+f}\bs{b}_f)$. We rewrite this last equation as \[A_i^\pi\cdot\begin{pmatrix} v_1 \\ v_2 \\ \vdots \\ v_m \end{pmatrix}-\begin{pmatrix} e_i\cdot\bs{b}_1 \\ e_i\cdot\bs{b}_2 \\ \vdots \\ e_i\cdot\bs{b}_f \end{pmatrix}\cdot \begin{pmatrix} v_{m+1} \\ v_{m+2} \\ \vdots \\ v_{m+f} \end{pmatrix} = -B_i^\pi\] i.e. \[C_i^\pi\cdot\v=-B_i^\pi.\]
	
	Let $p_m\colon\Z^{m+f}\rightarrow\Z^m$ denote projection onto the first $m$ coordinates. We can then express the set of all $\w$ that satisfy \eqref{dotprod3} as the following \[X^\pi= p_m\left(\bigcap_{i=1}^n\left\lbrace\v\in\Z^{m+f}\,\middle\vert\, C_i^\pi\cdot\v=-B_i^\pi\right\rbrace\right)\cap\N^m.\] Note that $\{\v\in\Z^{m+f}\mid C_i^\pi\cdot\v=-B_i^\pi\}$ is an elementary set for each $i$, and therefore the intersection is a (basic) polyhedral set. Therefore $X^\pi$ is also a polyhedral set. It corresponds to all words in $W^\pi$ which represent elements of the subgroup $H$. We wish to find unique minimal-weight representatives for these elements.
	
	Recall that the set $U^\pi\subset W^\pi$ is a set of minimal-weight representatives for the elements of $\overline{W^\pi}$. The intersection $\overrightarrow{U^\pi}\cap X^\pi$ consists of a unique, minimal-weight vector representing each element of $H$ that lies in $\overline{W^\pi}$ (and is not represented by a word with another pattern, of smaller weight). Since this intersection is polyhedral, it grows rationally with respect to the generators of $G$. So we have finitely many rationally growing sets whose (disjoint) union corresponds to a language of unique, minimal-weight representatives for $H$. The relative growth series of $H$ is the sum of the growth series of these individual sets, so is itself rational. Note that we have not had to consider overlaps as the sets $U^\pi$ are already non-overlapping.
\end{proof}

\begin{corollary}\label{subgroupunions}
	Finite unions of subgroups of virtually abelian groups have rational relative growth.
\end{corollary}
\begin{proof}
	 First note that the growth series of a finite \emph{disjoint} union of subsets of $G$ is simply the sum of their individual growth series, and for subsets $A\subset B\subseteq G$, the growth series of $B\setminus A$ is the difference of their individual growth series. We will induct on the number of subgroups in the finite union.
	
	For subgroups $H_1$ and $H_2$, let $I=H_1\cap H_2$. Then we can express their union as a disjoint union of three subsets, \[H_1\cup H_2 = I\cup (H_1\setminus I)\cup (H_2\setminus I).\] Since $H_1$, $H_2$, and $I$ are all subgroups, they have rational relative growth, so each term in the above expression does also, and so $H_1\cup H_2$ has rational growth series.
	
	Now assume the union of $k$ subgroups has rational growth. Consider a union of $k+1$ subgroups: \[\bigcup_{i=1}^{k+1}H_i=\bigcup_{i=1}^k H_i\cup H_{k+1}.\]
	Let $J=\left(\bigcup_{i=1}^k H_i\right) \cap H_{k+1}$. Then we have a disjoint union \[\bigcup_{i=1}^{k+1}H_i=J\cup\left(\bigcup_{i=1}^kH_i\right)\setminus J\cup\left(H_{k+1}\setminus J\right)\] and so if $J$ has rational growth then $\bigcup_{i=1}^{k+1}H_i$ has rational growth. But we can write $J$ as the union of $k$ subgroups: \[J=\bigcup_{i=1}^k(H_i\cap H_{k+1})\] so it has rational growth by the inductive hypothesis.
\end{proof}

%%%%%%%%%%%%%%%%%%%%%%%%%%%%%%%%%%%%%%%%%%%%%%%%%%%%%%%%%%%%%%%%%%%%%%%%%%%%%%%%%%%%%%%%%%%%%%%%%%%%%
%%%%%%%%%%%%%%%%%%%%%%%%%%%%%%%%%%%%%%%%%%%%%%%%%%%%%%%%%%%%%%%%%%%%%%%%%%%%%%%%%%%%%%%%%%%%%%%%%%%%%
%%%%%%%%%%%%%%%%%%%%%%%%%%%%%%%%%%%%%%%%%%%%%%%%%%%%%%%%%%%%%%%%%%%%%%%%%%%%%%%%%%%%%%%%%%%%%%%%%%%%%
%%%%%%%%%%%%%%%%%%%%%%%%%%%%%%%%%%%%%%%%%%%%%%%%%%%%%%%%%%%%%%%%%%%%%%%%%%%%%%%%%%%%%%%%%%%%%%%%%%%%%
%%%%%%%%%%%%%%%%%%%%%%%%%%%%%%%%%%%%%%%%%%%%%%%%%%%%%%%%%%%%%%%%%%%%%%%%%%%%%%%%%%%%%%%%%%%%%%%%%%%%%
%%%%%%%%%%%%%%%%%%%%%%%%%%%%%%%%%%%%%%%%%%%%%%%%%%%%%%%%%%%%%%%%%%%%%%%%%%%%%%%%%%%%%%%%%%%%%%%%%%%%%
%%%%%%%%%%%%%%%%%%%%%%%%%%%%%%%%%%%%%%%%%%%%%%%%%%%%%%%%%%%%%%%%%%%%%%%%%%%%%%%%%%%%%%%%%%%%%%%%%%%%%

\section{Coset Growth series}\label{sec:coset}
In this section we demonstrate the following:

\begin{theorem}\label{thm:cosetfull}
	Let $G$ be a virtually abelian group with any choice of finite, weighted generating set $S$, and $H$ any subgroup of $G$. Then the set of right cosets $H\backslash G$ has rational weighted growth series with respect to $S$.
\end{theorem}

This is proved in two parts. Firstly, we generalise the main result of \cite{Benson} to show that if a subgroup $F$ is contained in a finite index abelian subgroup of $G$ then there exists a language of minimal representatives for the cosets of $F$ in $G$ that grows rationally.

For a general subgroup $H$, its intersection with the finite index abelian subgroup of $G$ is abelian, and of finite index in $H$. We then consider the language of minimal representatives for the cosets of this intersection subgroup in $G$, and from these we choose a language of representatives for the cosets of $H$. We show that this language can be chosen so that it grows rationally.

We prove the following special case of Theorem \ref{thm:cosetfull}.
\begin{theorem}\label{thm:cosetpart}
	Let $G$ be a virtually abelian group, with any choice of finite weighted generating set $S$, and with finite index normal subgroup $\Z^n$. If $F$ is a subgroup of $G$ contained in $\Z^n$, then the weighted growth of $F\backslash G$ with respect to $S$ is rational.
\end{theorem}
Theorem 1.2 of \cite{Benson} is the case where $F$ is the trivial group. The proof given here follows a similar structure to that in \cite{Benson}. For each pattern $\pi\in P$, we establish an ordering on the words of $W^\pi$ which respects their weight, and use this to find a polyhedral set of minimal representatives for the cosets that intersect $\overline{W^\pi}$. We then show that the overlaps between these sets are also polyhedral, so can be removed while preserving rationality. Note that if $F$ has finite index in $G$, $F\backslash G$ is finite and so the growth series is a polynomial, and so trivially rational. From now on we assume that $F$ has infinite index in $G$, and hence $[\Z^n\colon F]$ is also infinite.

 We first establish a criterion for when two words represent elements of the same $F$-coset. We will again use the vectors $A_1^\pi,\ldots,A_n^\pi$ (of dimension $m(\pi)$), $A_1^\mu,\ldots,A_n^\mu$ (of dimension $m(\mu)$) and integers $B_1^\pi,\ldots,B_n^\pi$, and $B_1^\mu\ldots,B_n^\mu$ defined in section \ref{sec:prelim}.

\begin{proposition}\label{prop:cosetcriterion}
Let $v\in W^\pi$, $w\in W^\mu$, for some patterns $\pi$,$\mu$. Let $F<\Z^n$ be of rank $f\leq n$, with basis $\{\bs{b}_1,\bs{b}_2,\ldots,\bs{b}_f\}\subset\Z^n$. Then $\overline{v}$ and $\overline{w}$ are in the same coset of $F$ in $G$ if and only if
\begin{enumerate}
	\item $\overline{\pi}$ and $\overline{\mu}$ are in the same coset of $\Z^n$ in $G$, and\label{item:cosetcriterion1}
	\item there exist integers $a_1,\ldots,a_f$ such that \[A_i^\pi\cdot\v+B_i^\pi-\left(A_i^\mu\cdot\w+B_i^\mu\right)=e_i\cdot\left(\sum_{j=1}^fa_j\bs{b}_j\right)\] for each $1\leq i\leq n$ (where $e_i$ as usual denotes the $i$th standard basis vector in $\Z^n$).\label{item:cosetcriterion2}
\end{enumerate}
\end{proposition}
\begin{proof}
 Recall (see \eqref{eq:wbar}) that the group elements represented by $v$ and $w$ are given by
\begin{equation*}
\overline{v}=\left\lbrace(A_1^\pi\cdot \v,A_2^\pi\cdot \v,\ldots, A_n^\pi\cdot \v)^T+(B_1^\pi,B_2^\pi,\ldots,B_n^\pi)^T\right\rbrace t_\pi\in G
\end{equation*}
and
\begin{equation*}
\overline{w}=\left\lbrace(A_1^\mu\cdot \w,A_2^\mu\cdot \w,\ldots, A_n^\mu\cdot \w)^T+(B_1^\mu,B_2^\mu,\ldots,B_n^\mu)^T\right\rbrace t_\mu\in G
\end{equation*}
respectively, where $t_\pi$ and $t_\mu$ are the chosen representatives for the cosets $\Z^n\overline{\pi}$ and $\Z^n\overline{\mu}$.

Now two words $v,w$ represent elements of the same $F$-coset if and only if $\overline{v}(\overline{w})^{-1}\in H$. This is equivalent to the existence of integer coefficients $a_1,\ldots,a_f$ such that \[\overline{v}(\overline{w})^{-1}=\sum_{j=1}^fa_j\bs{b}_j.\]

Suppose that our words $v$ and $w$ do represent the same $F$-coset. Since $Fg\subset\Z^ng$ for any $g\in G$, $v$ and $w$ represent the same $\Z^n$-coset, so $t_\pi=t_\mu$, which is precisely condition \eqref{item:cosetcriterion1}.

We have
\begin{align*}
\overline{v}(\overline{w})^{-1} &= \left\lbrace(A_1^\pi\cdot \v,A_2^\pi\cdot \v,\ldots, a_n^\pi\cdot \v)^T+(B_1^\pi,B_2^\pi,\ldots,B_n^\pi)^T\right\rbrace t_\pi \\ &\hspace{50pt}\cdot t_\mu^{-1}\left\lbrace-\left(A_1^\mu\cdot \w,A_2^\mu\cdot \w,\ldots, A_n^\mu\cdot \w\right)^T-\left(B_1^\mu,B_2^\mu,\ldots,B_n^\mu\right)^T\right\rbrace \\
&= (A_1^\pi\cdot\v+B_1^\pi-A_1^\mu\cdot\w-B_1^\mu,\ldots,A_n^\pi\cdot\v+B_n^\pi-A_n^\mu\cdot\w-B_n^\mu)^T.
\end{align*}
Thus
\begin{equation*}
(A_1^\pi\cdot\v+B_1^\pi-A_1^\mu\cdot\w-B_1^\mu,\ldots,A_n^\pi\cdot\v+B_n^\pi-A_n^\mu\cdot\w-B_n^\mu)^T = \sum_{j=1}^fa_j\bs{b}_j\in\Z^n,
\end{equation*}
which is equivalent to condition \eqref{item:cosetcriterion2}.

Conversely, if $v$ and $w$ satisfy the conditions, then $t_\pi=t_\mu$ and so 
\begin{equation*}
\overline{v}(\overline{w})^{-1}=(A_1^\pi\cdot\v+B_1^\pi-A_1^\mu\cdot\w-B_1^\mu,\ldots,A_n^\pi\cdot\v+B_n^\pi-A_n^\mu\cdot\w-B_n^\mu)^T = \sum_{j=1}^fa_j\bs{b}_j
\end{equation*}
i.e. $\overline{v}(\overline{w})^{-1}\in F$.
\end{proof}

We note the following special case, when $\pi=\mu$.
\begin{corollary}\label{cor:cosetcriterion}
	Let $v,w\in W^\pi$, $F<\Z^n$ with rank $f$ and basis $\{\bs{b}_1,\bs{b}_2,\ldots,\bs{b}_f\}\subset\Z^n$. Then $\overline{v}$ and $\overline{w}$ are in the same $F$-coset if and only if there exist integers $a_1,\ldots a_f$ so that \[A_i^\pi\cdot(\v-\w)=e_i\cdot\sum_{j=1}^fa_j\bs{b}_j\] for each $1\leq i\leq n$.
\end{corollary}

In what follows we will need a version of Dickson's Lemma.
\begin{definition}
Let $I\subseteq\{1,2,\ldots,m\}$ for some $m$, and define \[Q_I=\{x\in\Z^m\mid e_i\cdot x\geq0,\text{ if }i\in I,~ e_i\cdot x\leq 0,\text{ if }i\notin I\},\] which is a closed orthant of $\Z^m$. Define the `coordinate ordering' on $Q_I$ as follows: $x\leq_I y$ if and only if $e_i\cdot x\leq e_i\cdot y$ for $i\in I$ and $e_i\cdot x\geq e_i\cdot y$ for $i\notin I$.
\end{definition}
\begin{lemma}[Lemma A of \cite{Dickson}]\label{lem:Dickson}
Let $X\subseteq Q_I$ for some $I\subseteq\{1,2,\ldots,m\}$. Then there exists a finite subset $Y\subset X$ such that for all $x\in X$, there exists $y\in Y$ with $y\leq_I x$.
\end{lemma}
The result is more commonly stated for the case $I=\{1,2,\ldots,m\}$ (and so $Q_I=\N^m$) but this more general version follows directly from symmetry.

We now build a polyhedral set of minimal weight coset representatives, for each pattern $\pi$. The following is a modification of the arguments in Section $6$ of \cite{Benson}.

\begin{proposition}\label{prop:VHpi}
	Fix a pattern $\pi\in P$, and with $F$ an infinite index subgroup of $\Z^n\lhd G$, consider the set $(F\backslash G)^\pi$ of right cosets which contain an element represented by a word in $W^\pi$, i.e. \[(F\backslash G)^\pi=\{Fg\in F\backslash G\mid Fg\cap\overline{W^\pi}\neq\emptyset\}.\] Then there exists a set $V_F^\pi\subset W^\pi$ consisting of minimal-weight (amongst $W^\pi$) representatives for every coset in $(F\backslash G)^\pi$, with the property that $\overrightarrow{V_F^\pi}$ is a polyhedral set.
\end{proposition}
\begin{proof}
	We will define an ordering on words in $W^\pi$ that is consistent with the weight ordering, and yields a unique minimal representative for each coset in $(F\backslash G)^\pi$.
	
	As above, fix a basis $\{\bs{b}_1,\ldots,\bs{b}_f\}$ for $F$. Recall the vectors $A_1^\pi,\ldots,A_{n+1}^\pi\in\Z^m$. Choose standard basis vectors $A_{n+2}^\pi,\ldots,A_K^\pi$ so that the set $\{A_{n+1}^\pi,\ldots,A_K^\pi\}$ spans $\Z^m$. Define an order on the words of $W^\pi$ as follows. We will write $v\leq_\pi w$ if and only if either $v=w$ or there exist integers $a_1,\ldots,a_f$, and $0\leq i\leq K-n$ such that
	\begin{align*}
	A_k^\pi\cdot(\v-\w)&=e_k\cdot\sum_{j=1}^f a_j\bs{b}_j,\,\text{ for } 1\leq k\leq n,\\
	A_k^\pi\cdot(\v-\w)&=0,\,\text{ for } n+1\leq k< n+i,\,\text{ and}\\
	A_{n+i}^\pi\cdot(\v-\w)&>0.
	\end{align*}
	We show that this is a partial order on $W^\pi$. Firstly, note that the ordering is reflexive by definition. Next, we show transitivity. Suppose $u,v,w\in W^\pi$ with $u\leq_\pi v\leq_\pi w$. If $u=v$ or $v=w$ then clearly $u\leq_{\pi}w$, so suppose $u\neq v\neq w$. So we have integers $a_1,\ldots,a_f,a_1'\ldots,a_f'$ and $0\leq i,i'\leq K-n$ with 
	\begin{align*}
	A_k^\pi\cdot(\u-\v)&=e_k\cdot\sum_{j=1}^f a_j\bs{b}_j,\,\text{ for } 1\leq k\leq n,\\
	A_k^\pi\cdot(\u-\v)&=0,\,\text{ for } n+1\leq k< n+i,\,\text{ and}\\
	A_{n+i}^\pi\cdot(\u-\v)&>0
	\end{align*}
	and 
	\begin{align*}
	A_k^\pi\cdot(\v-\w)&=e_k\cdot\sum_{j=1}^f a'_j\bs{b}_j,\,\text{ for } 1\leq k\leq n,\\
	A_k^\pi\cdot(\v-\w)&=0,\,\text{ for } n+1\leq k< n+i',\,\text{ and}\\
	A_{n+i'}^\pi\cdot(\v-\w)&>0.
	\end{align*}
	Then for $0\leq k\leq n$ we have \[A_k^\pi\cdot(\u-\w)=A_k^\pi\cdot(\u-\v)+A_k^\pi\cdot(\v-\w)=e_k\cdot\sum_{j=1}^f (a_j+a_j')\bs{b}_j\] by linearity of the scalar product. Furthermore, $A_k^\pi\cdot(\u-\w)=0$ for $n+1\leq k<\min(i,i')$ and $A_k^\pi\cdot(\u-\w)>0$ for $k=\min(i,i')$, and thus $u\leq_\pi w$, so $\leq_\pi$ is transitive. For antisymmetry, note that if $v\leq_\pi w$ and $w\leq_\pi v$ then either $v=w$ or we must have $A^\pi_k\cdot(\v-\w)=A^\pi_k\cdot(\w-\v)=0$ for each $n+1\leq k\leq K$, and since the corresponding $A^\pi_k$s span $\Z^m$, we must have $\v=\w$, i.e. $v=w$. Thus the order is a well-defined partial order (although not a total order).
	
	If we restrict ourselves to the words representing a single $H$-coset, this becomes a well-ordering. To see this, suppose we have an infinite descending chain of words in $W^\pi$ that represent the same coset: 
	\begin{equation*}
		w_1\geq_\pi w_2\geq_\pi\cdots.
	\end{equation*}
	Since $A_k^\pi\in\N^m$ for $k>n$ and $\w_i\in\N^m$ for all $i$, the sequences \[A_k^\pi\cdot \w_1\geq A_k^\pi\cdot\w_2\geq\cdots\] for each $k>n$ consist of non-negative integers and so must stabilize, say after $i_k$ steps. Let $i_{\max}=\max_{n<k\leq K}(i_k)$. Then $A_k^\pi\cdot \w_{i_{\max}}=A_k^\pi\cdot \w_{i_{\max}+j}$ for any positive integer $j$, and all $n<k\leq K$. Therefore since the vectors $A_k^\pi$ for $k>n$ span $\Z^m$, $w_{i_{\max}}=w_{i_{\max+j}}$ and the sequence stabilizes. Note that two words representing the same coset can always be compared under $\leq_\pi$, since Corollary \ref{cor:cosetcriterion} implies there are integers $a_1,\ldots,a_f$ satisfying the definition. Thus there is a unique $\leq_\pi$-minimal word in $W^\pi$ that represents each coset in $(F\backslash G)^\pi$.
	
	Note that if $v\leq_\pi w$ then $A_{n+1}^\pi\cdot\v\leq A_{n+1}^\pi\cdot\w$ and so $\omega(v)\leq\omega(w)$. Thus the unique $\leq_\pi$-minimal element in $W^\pi$ that represents a given $F$-coset is also a minimal weight coset representative (amongst $W^\pi$). Let $V_F^\pi$ denote the set of all $\leq_\pi$-minimal representatives in $W^\pi$, that is
	\begin{equation}\label{eq:VHpi}
	V_F^{\pi}:=\left\{w\in W^\pi\,\middle\vert\,\text{If }v\in W^\pi\text{ s.t. }\overline{v}\in F\overline{w}\text{ then }w\leq_\pi v\right\}.
	\end{equation} To finish the proof, we need to show that $V_F^\pi$ corresponds to a polyhedral set in $\Z^m$.
	
	An element $\tau\in\Z^m$ will be called a \emph{translation} (with respect to $F$, $\pi$) if there exist integers $a_1,\ldots,a_f$ and $0\leq i\leq K-n$ such that $A_k^\pi\cdot\tau=e_k\cdot\sum_{j=1}^f a_j\bs{b}_j$ for $1\leq k\leq n$ and \[A_{n+1}^\pi\cdot\tau=\cdots=A_{n+i-1}^\pi\cdot\tau=0,~A_{n+i}^\pi\cdot\tau>0.\] Let $\cT$ denote the set of all such translations. Suppose $v,w\in W^\pi$. Then it is clear that $\overline{w}\in F\overline{v}$ with $w\leq_\pi v$ if and only if there exists some $\tau\in\cT$ with $\v=\w+\tau$. The set $\cT$ is contained in $\Z^m$. Consider $\cT\cap Q_I$ for some $I\subseteq\{1,\ldots,m\}$. By Lemma \ref{lem:Dickson}, there exists a finite set $\cT_I\subset\cT\cap Q_I$ such that each element $\tau\in\cT\cap Q_I$ has a bound $\tau_0\in\cT_I$ such that $\tau_0\leq_I\tau$. Let $\cT_{\min}=\bigcup_I \cT_I$, the union of the minimal translations across all orthants. We now claim that \[\overrightarrow{V_F^\pi}=\N^m\setminus\bigcup_{\tau\in\cT_{\min}}(\tau+\N^m).\] It is not hard to see that this is a polyhedral set, which proves the Proposition.
	
	To see the claim, first suppose that $v\in V_F^\pi$ but $\v\notin\N^m\setminus\bigcup_{\tau\in\cT_{\min}}(\tau+\N^m)$. So $\v\in\tau+\N^m$ for some $\tau\in\cT_{\min}$, i.e. $\v=\tau+\w$ for some $\w\in\N^m$. This implies that there is some $w\in W^\pi$ which shares an $F$-coset with $v$ such that $w\leq_\pi v$. But this implies that $v$ is not minimal, contradicting the assumption that $v\in V^\pi_F$.
	
	Conversely, suppose that $\v\in\N^m\setminus\bigcup_{\tau\in\cT_{\min}}(\tau+\N^m)$ and $v\notin V^\pi_F$. So there exists some $\tau\in\cT$ and $\v_0\in V^\pi_F$, with $\v=\v_0+\tau$ and $\overline{v_0}\in F\overline{v}$. In other words $\v-\v_0$ is a translation. Choose $I$ so that $\v-\v_0\in Q_I$, and then $\tau_0\in\cT_I$ so that $\tau_0\leq\v-\v_0$. We claim that $\v-\tau_0\in\N^m$, i.e. $\v\in\tau_0+\N^m$, contradicting our assumption. Indeed, for $i\in I$ we have $e_i\cdot\tau_0\leq e_i\cdot\tau$, so $e_i\cdot(\v-\tau_0)\geq e_i\cdot(\v-\tau)=e_i\cdot\v_0\geq 0$, and for $i\notin I$ we have $e_i\cdot(\v-\tau_0)\geq e_i\cdot\tau_0\geq 0$. So for all $i$, $e_i\cdot(\v-\tau_0)\geq 0$, and hence $\v-\tau\in\N^m$ as claimed.
\end{proof}

Since $Fg\subset\Z^ng$ for any $g\in G$, any two words that represent the same $F$-coset must lie in the same $\Z^n$-coset. So we consider each $\Z^n$-coset separately. Section \ref{sec:prelim} tells us that for a given $\Z^n$-coset, say $\Z^nt$ for $t\in T$, there is a finite set of patterns $P_t$, over the extended generating set $\wt{S}$, whose patterned sets contain representatives for all the elements of the coset (and no other cosets). We take the corresponding polyhedral sets $\overrightarrow{V^\pi_F}$ from equation \eqref{eq:VHpi} for each $\pi\in P_t$, and combine them to find a language of representatives for the $F$-cosets within $\Z^nt$. We may have pairs of words with different patterns that both represent the same coset, but we only wish to count the minimal one. To prove Theorem \ref{thm:cosetpart}, we show that these overlaps between the $\overrightarrow{V^\pi_F}$s are polyhedral, so can be removed without losing rationality.

\begin{definition}
Let $\pi,\mu\in P_t$ for some $t\in T$. Define the set $R^{\pi,\mu}$ (resp. $R_*^{\pi,\mu}$) consisting of all those elements of $V_F^\pi$ where there is an element of $V_F^\mu$ of strictly smaller (resp. equal) weight that represents the same coset:
\begin{align*}
R^{\pi,\mu}:=&\{v\in V_F^\pi\mid\exists u\in V_F^\mu,~\overline{u}\in F\overline{v},~\omega(u)<\omega(v)\}\\
R_*^{\pi,\mu}:=&\{v\in V_F^\pi\mid\exists u\in V_F^\mu,~\overline{u}\in F\overline{v},~\omega(u)=\omega(v)\}.
\end{align*}
\end{definition}

We need to discard all of $R^{\pi,\mu}$ for every pair $\pi\neq\mu$, since we only want minimal words. If there exist two of more minimal weight representatives for the same coset with equal weight, we must choose exactly one and discard the rest. We make the following definition.

\begin{definition}\label{def:UHpi}
Pick a total order on the finite set $P_t$, denoted $\pi_1<\pi_2<\cdots$. Let
\begin{equation*}
U_F^{\pi_k}:=V_F^{\pi_k}\setminus\left[\bigcup_{i\neq k}R^{\pi_k,\pi_i}\cup\bigcup_{j<k}R^{\pi_k,\pi_j}_*\right].
\end{equation*}
\end{definition}

So $U_F^{\pi_k}$ consists of those minimal-weight coset representatives in $W^{\pi_k}$ where there are no representatives of the same coset with smaller weight and a different pattern, and wherever there are multiple representatives with equal weight we choose based on the order on $P_t$.

\begin{proof}[Proof of Theorem \ref{thm:cosetpart}]
We claim that $\overrightarrow{U_F^{\pi_k}}\subset\N^{m(\pi)}$ is a polyhedral set for each $\pi_k$. Then the disjoint union of the sets $U_F^{\pi_k}$ for each $\pi_k\in P_t$, and each $t\in T$, is a finite disjoint union of rationally growing languages, forming a set of minimal weight representatives for the cosets $F\backslash G$, which will prove the Theorem.

To prove that $\overrightarrow{U_F^{\pi_k}}$ is polyhedral, it is enough to show that $\overrightarrow{R^{\pi,\mu}}$ and $\overrightarrow{R_*^{\pi,\mu}}$ are polyhedral for any $\pi,\mu$, since $\overrightarrow{U_F^{\pi_k}}$ is then obtained from polyhedral sets via finite unions and set complement, so is itself polyhedral.

Let $\mathbf{1}_j\in\Z^{2n+2+f}$ be the vector with a $1$ at the $j$th entry and zeroes everywhere else. Define the vectors
\begin{equation*}
E_i=\mathbf{1}_i-\mathbf{1}_{i+n+1}+\begin{pmatrix}0\\ \vdots \\ 0 \\ -e_i\cdot\bs{b}_1 \\ -e_i\cdot\bs{b}_2\\ \vdots \\-e_i\cdot\bs{b}_f\end{pmatrix}
\begin{tabular}{l}
$\left.\lefteqn{\phantom{\begin{matrix} 0\\ \vdots \\ 0 \end{matrix}}}\right\}2n+2$ zeroes \\
$\left.\lefteqn{\phantom{\begin{matrix} -e_i\cdot\bs{b}_1 \\ -e_i\cdot\bs{b}_2\\ \vdots \\ -e_i\cdot\bs{b}_f \end{matrix}}}\right\}f$ rows \\
\end{tabular}
\end{equation*}
for each $1\leq i\leq n$, and let $E_{n+1}=\mathbf{1}_{n+1}-\mathbf{1}_{2n+2}$. Then define the polyhedral sets \[\Phi:=\bigcap_{i=1}^n\left\{\phi\in\Z^{2n+2+f}\mid\phi\cdot E_i=0\right\}\cap\left\{\phi\in\Z^{2n+2+f}\mid\phi\cdot E_{n+1}>0\right\}\] and \[\Phi_*:=\bigcap_{i=1}^n\left\{\phi\in\Z^{2n+2+f}\mid\phi\cdot E_i=0\right\}\cap\left\{\phi\in\Z^{2n+2+f}\mid\phi\cdot E_{n+1}=0\right\}.\]

Let $\cE^\pi\colon\overrightarrow{V_F^\pi}\rightarrow\Z^{n+1}$ denote the integral affine transformation \[\cE^\pi\colon\w\mapsto\begin{pmatrix} A_1^\pi\cdot\w+B_1^\pi \\ A_2^\pi\cdot\w+B_2^\pi \\ \vdots \\ A_{n+1}^\pi\cdot\w+B_{n+1}^\pi \end{pmatrix}\] for any pattern $\pi$ (and write $(\cE^\pi)^{-1}X$ for the preimage in $\overrightarrow{V_F^\pi}$ of any $X\subseteq\Z^{n+1}$). For any $k'>k$, write $p_k\colon\Z^{k'}\rightarrow\Z^k$ for the projection onto the first $k$ coordinates. We will show that
\begin{equation}\label{eq:polyhedraloverlap}
\overrightarrow{R^{\pi,\mu}}=(\cE^\pi)^{-1}p_{n+1}\left[\left(\cE^\pi(\overrightarrow{V_F^\pi})\times\cE^\mu(\overrightarrow{V_F^\mu})\right)\cap p_{2n+2}(\Phi)\right],
\end{equation}
which is a polyhedral set since projection is an affine transformation. Indeed, suppose that $v\in R^{\pi,\mu}$. So there exists $u\in V_F^\mu$ such that $\overline{u}\in F\overline{v}$ and $\omega(u)<\omega(v)$. By Corollary \ref{cor:cosetcriterion}, there exist integers $a_1,\ldots,a_f$ such that \[A_i^\pi\cdot\v+B_i^\pi-(A_i^\mu\cdot\u+B_i^\mu)=e_i\cdot\sum_{j=1}^f a_j\bs{b}_j\] for each $1\leq i\leq n$, and $A_{n+1}^\pi\cdot\v+B_{n+1}^\pi>A_{n+1}^\mu+B_{n+1}^\mu$. In other words \[\begin{pmatrix} A_1^\pi\cdot\v+B_1^\pi \\ A_2^\pi\cdot\v+B_2^\pi \\ \vdots \\ A_{n+1}^\pi\cdot\v+B_{n+1}^\pi \\ A_1^\mu\cdot\u+B_1^\mu \\ A_2^\mu\cdot\u+B_2^\mu \\ \vdots \\ A_{n+1}^\mu\cdot\u+B_{n+1}^\mu \\ a_1 \\ \vdots \\ a_f \end{pmatrix}\in\Phi,\] and hence $\left(\cE^\pi(\v),\cE^\mu(\u)\right)\in p_{2n+2}(\Phi)$, so $\v$ is contained in the right hand side of \eqref{eq:polyhedraloverlap}.

Conversely, let $\v\in\N^{m(\pi)}$ be contained in the right hand side of \eqref{eq:polyhedraloverlap}. Thus \[\cE^\pi(\v)\in p_{n+1}\left[\left(\cE^\pi(\overrightarrow{V_F^\pi})\times\cE^\mu(\overrightarrow{V_H^\mu})\right)\cap p_{2n+2}(\Phi)\right]\] so there exists $\bs{z}\in\Z^{n+1}$ with  \[\left(\cE^\pi(\v),\bs{z}\right)\in\left(\cE^\pi(\overrightarrow{V_F^\pi})\times\cE^\mu(\overrightarrow{V_F^\mu})\right)\cap p_{2n+2}(\Phi).\] That is, there exists $u\in V_F^\mu$ with $\bs{z}=(A_1^\mu\cdot\u+B_1^\mu,\ldots,A_{n+1}^\mu\cdot\u+B_{n+1}^\mu)^T$ and there exist integers $a_1,\ldots,a_f$ such that $(\cE^\pi(\v),\bs{z},a_1,\ldots,a_f)^T\in\Phi$, and together this means that \[\begin{pmatrix} A_1^\pi\cdot\v+B_1^\pi \\ A_2^\pi\cdot\v+B_2^\pi \\ \vdots \\ A_{n+1}^\pi\cdot\v+B_{n+1}^\pi \\ A_1^\mu\cdot\u+B_1^\mu \\ A_2^\mu\cdot\u+B_2^\mu \\ \vdots \\ A_{n+1}^\mu\cdot\u+B_{n+1}^\mu \\ a_1 \\ \vdots \\ a_f \end{pmatrix}\in\Phi.\] From the definition of $\Phi$, this implies that \[A_i^\pi\cdot\v+B_i^\pi-\left(A_i^\mu\cdot\u+B_i^\mu\right) = e_i\cdot\sum_{j=1}^fa_j\bs{b}_j\] for each $1\leq i\leq n$ and $A_{n+1}^\pi\cdot\v+B_{n+1}^\pi>A_{n+1}^\mu\cdot\u+B_{n+1}^\mu$. So $v\in V_F^\pi$ and there exists $u\in V_F^\mu$ with $\overline{u}\in F\overline{v}$ and $\omega(u)<\omega(v)$, i.e. $v\in R^{\pi,\mu}$, and so $\overrightarrow{R^{\pi,\mu}}$ has the polyhedral form \eqref{eq:polyhedraloverlap} as claimed.

In an exactly analogous way, \[\overrightarrow{R_*^{\pi,\mu}}=(\cE^\pi)^{-1}p_{n+1}\left[\left(\cE^\pi(\overrightarrow{V_F^\pi})\times\cE^\mu(\overrightarrow{V_F^\mu})\right)\cap p_{2n+2}(\Phi_*)\right].\]

\end{proof}

We now use the previous result to prove Theorem \ref{thm:cosetfull}, that is, to show that for an arbitrary subgroup $H\leq G$, the set of right cosets $H\backslash G$ has rational growth with respect to any choice of weighted generating set for $G$. The proof relies on the understanding of the structure of subgroups in Lemma \ref{lem:subgroupcriteria} and the rationality of coset growth for free abelian subgroups in Theorem \ref{thm:cosetpart}.
\begin{proof}[Proof of Theorem \ref{thm:cosetfull}]
First, we consider the coset structure of $G$. Let $\Z^n$ be the maximal free abelian normal subgroup of $G$. Then we have $H\cap\Z^n\lhd H$, with finite index $c\leq d$. Fix a choice of transversal $\{h_1,\ldots,h_c\}$ for the cosets $(H\cap\Z^n)\backslash H$. As in Lemma \ref{lem:subgroupcriteria}, we extend this to a transversal for $\Z^n\backslash G$, write $T=\{t_1,\ldots,t_d\}\supseteq\{h_1,\ldots,h_c\}$. Suppose that the free abelian group $H\cap\Z^n$ has rank $f$, and fix a basis $\{\bs{b}_1,\ldots,\bs{b}_f\}\subset\Z^n$.

Consider a coset $Hg\in H\backslash G$. Following the above discussion, we may decompose $H$ as the finite union of its $H\cap\Z^n$-cosets. We may also write $g=(g_1,\ldots,g_n)^Tt$ for some $g_i\in\Z$ and $t\in T$. Therefore \[Hg=\left(\bigcup_{j=1}^c(H\cap\Z^n)h_j\right)\begin{pmatrix}g_1 \\ \vdots \\ g_n\end{pmatrix}t.\] Since $H\cap\Z^n$ is also a subgroup of $G$, contained in $\Z^n$, Theorem \ref{thm:cosetpart} provides a minimal weight representative for each coset of $H\cap\Z^n$ in $G$ of the form $(H\cap\Z^n)h_jg$, and therefore a collection of $c$ candidates for a minimal-weight representative for $Hg$, since the minimal-weight representative for $Hg$ is one of the $c$ minimal-weight representatives for the cosets $(H\cap\Z^n)h_jg$. We will express these candidates as $c$-tuples of words, (whose patterns together make $c$-fold patterns, see Section \ref{sec:N-foldpatterns}), and show that they correspond to polyhedral sets, from which rationality will follow. 

Fix a $c$-fold pattern $\bs{\pi}=(\pi_1,\ldots,\pi_c)\in P^c$ (recalling the definition of $P$ from Definition \ref{def:finitepatterns}), and define
\begin{align}\label{eq:Vpi}
V(\bs{\pi})=\left\{\left(w^{(1)},\ldots,w^{(c)}\right)\right.&\left.\in U^{\pi_1}_{H\cap\Z^n}\times\cdots\times U^{\pi_c}_{H\cap\Z^n}\,\middle\vert\,\right. \\ &\left. \exists g\in G\text{ s.t. }\overline{w^{(j)}}\in(H\cap\Z^n)h_jg,~1\leq j\leq c\right\},\nonumber
\end{align}
where $U_{H\cap\Z^n}^{\pi_k}$ is the set of minimal representatives for the cosets of $H\cap\Z^n$ in $G$ as defined in Definition \ref{def:UHpi}, where $H\cap\Z^n$ plays the role of $F$. Each element of $V(\bs{\pi})$ consists of a $c$-tuple of candidates for a minimal weight representative of the coset $Hg$. By Lemma \ref{lem:minimaltuplerep}, if $\overrightarrow{V(\bs{\pi})}$ is polyhedral then there is a language $\cL_{\bs{\pi}}$ of minimal representatives for those cosets represented by $V(\bs{\pi})$, which grows rationally. Every element of $G$ has a minimal-weight representative with a pattern in $P$ (by definition of $P$), and so in particular every coset has a minimal-weight representative with pattern in $P$, and is therefore represented in some $V(\bs{\pi})$. Thus the union $\bigcup_{\bs{\pi}\in P^c}\cL_{\bs{\pi}}$ forms a language of minimal weight representatives for the set of cosets $H\backslash G$. Since $P^c$ is finite, this union is a finite union of polyhedral sets, and thus has rational growth.

We now show that $\overrightarrow{V(\bs{\pi})}$ is indeed a polyhedral set for each $\bs{\pi}\in P^c$, which will complete the proof. For some tuple in $V(\bs{\pi})$, consider the element $g$ as in equation \eqref{eq:Vpi}. This can be expressed as $(g_1,\ldots,g_n)^Tt$ for some $g_i\in\Z$ and $t\in T$.

So we can decompose $V(\bs{\pi})$  as a finite union $\bigcup_{t\in T}V(\bs{\pi},t)$ where 
\begin{align*}
V(\bs{\pi},t)=\left\{\left(w^{(1)},\ldots,w^{(c)}\right)\right.&\left.\in U^{\pi_1}_{H\cap\Z^n}\times\cdots\times U^{\pi_c}_{H\cap\Z^n}\,\middle\vert\, \right.\\
 &\left.\exists g\in\Z^nt\text{ s.t. }\overline{w^{(j)}}\in(H\cap\Z^n)h_jg,~1\leq j\leq c\right\}.
\end{align*}

We wish to write an element of $(H\cap\Z^n)h_j(g_1,\ldots,g_n)^Tt$ in the standard form defined in section \ref{sec:notation}. Recall from section \ref{sec:conjconstants} that for any element $s\in T$, we have a matrix $\Delta_s$ so that $s(a_1,\ldots,a_n)^Ts^{-1}=\Delta_s(a_1,\ldots,a_n)^T$ for any integers $a_i$. For two coset representatives $s,t\in T$, their product $st$ will not necessarily be in $T$. So let $x_{st}\in\Z^n$ and $\tau_{st}\in T$ be such that $st=x_{st}\tau_{st}$.

Now suppose that \[\gamma\in (H\cap\Z^n)h_j\begin{pmatrix}g_1 \\ g_2 \\ \vdots \\ g_n\end{pmatrix}t\subseteq H\begin{pmatrix}g_1 \\ g_2 \\ \vdots \\ g_n\end{pmatrix}t.\] So there exist integers $\lambda_1^{(j)},\ldots,\lambda_f^{(j)}$, so that 
\begin{align*}
\gamma&=\left(\sum_{k=1}^f\lambda_k^{(j)}\bs{b}_k\right)h_j\begin{pmatrix}g_1 \\ g_2 \\ \vdots \\ g_n\end{pmatrix}t\\
&=\left(\sum_{k=1}^f\lambda_k^{(j)}\bs{b}_k+\Delta_{h_j}\begin{pmatrix}g_1 \\ g_2 \\ \vdots \\ g_n\end{pmatrix}\right)h_jt\\
&=\left(\sum_{k=1}^f\lambda_k^{(j)}\bs{b}_k+\Delta_{h_j}\begin{pmatrix}g_1 \\ g_2 \\ \vdots \\ g_n\end{pmatrix}+x_{h_jt}\right)\tau_{h_jt},
\end{align*}
which is in the standard form. Thus for $w^{(j)}\in W^{\pi_j}$, we have $\overline{w^{(j)}}\in(H\cap\Z^n)h_jg$ for some $g=(g_1,\ldots,g_n)^Tt$ if and only if
\begin{enumerate}
	\item there exists $\lambda_k^{(j)}\in\Z$ for each $1\leq k\leq f$ such that
	\begin{equation*}
	\begin{pmatrix} A_1^{\pi_j}\cdot\overrightarrow{w^{(j)}} \\ A_2^{\pi_j}\cdot\overrightarrow{w^{(j)}} \\ \vdots \\ A_n^{\pi_j}\cdot\overrightarrow{w^{(j)}} \end{pmatrix} +\begin{pmatrix} B_1^{\pi_j} \\ B_2^{\pi_j} \\ \vdots \\ B_n^{\pi_j}\end{pmatrix} = \sum_{k=1}^f\lambda_k^{(j)}\bs{b}_k + \Delta_{h_j}\begin{pmatrix}g_1 \\ g_2 \\ \vdots \\ g_n\end{pmatrix} + x_{h_jt},
	\end{equation*}
	and
	\item $t_{\pi_j}=\tau_{h_jt}$,
\end{enumerate} 
(where $\overline{\pi_j}=(B_1^{\pi_j},\ldots,B_n^{\pi_j})^Tt_{\pi_j}$ as before). The first condition can be restated as follows. There exists $\lambda_k^{(j)}\in\Z$ for each $1\leq k\leq f$ such that
\begin{equation*}
A_i^{\pi_j}\cdot\overrightarrow{w^{(j)}}+B_i^{\pi_j}=e_i\cdot\left(\sum_{k=1}^f\lambda_k^{(j)}\bs{b}_k + \Delta_{h_j}\begin{pmatrix}g_1 \\ g_2 \\ \vdots \\ g_n\end{pmatrix} + x_{h_jt}\right)
\end{equation*}
for each $1\leq i\leq n$ and $1\leq j\leq c$. This can be re-written as
\begin{equation}\label{eq:cosetcondition}
A_i^{\pi_j}\cdot\overrightarrow{w^{(j)}} + \sum_{k=1}^f\lambda_k(-e_i\cdot\bs{b}_k)-e_i\cdot\Delta_{h_j}\begin{pmatrix}g_1 \\ g_2 \\ \vdots \\ g_n\end{pmatrix}=-B_i^{\pi_j}+e_i\cdot x_{h_jt}.
\end{equation}

Then we may rewrite $V(\bs{\pi},t)$ as follows.
\begin{align*}
V(\bs{\pi},t)=&\left\{\left(w^{(1)},\ldots,w^{(c)}\right)\in U^{\pi_1}_{H\cap\Z^n}\times\cdots\times U^{\pi_c}_{H\cap\Z^n}\,\middle\vert\,\exists(g_1,\ldots,g_n)^T\in\Z^n, \right.\\
&\left.\lambda_k^{(j)}\in\Z^n, \text{ s.t. each }w^{(j)}\text{ satisfies \eqref{eq:cosetcondition} for each }1\leq i\leq n\right\}
\end{align*}

For a fixed $\bs{\pi}$ and $t\in T$, we define vectors $L_i^j\in\Z^{m(\bs{\pi})+fc+n}$ as follows:
\begin{equation*}
L_i^j:=\begin{pmatrix} 0 \\ \vdots \\ 0 \\ A_i^{\pi_j} \\ 0 \\ \vdots \\ 0 \\ 0 \\ \vdots \\ 0 \\ -e_i\cdot\bs{b}_1 \\ \vdots \\ -e_i\cdot\bs{b}_f \\ 0 \\ \vdots \\ 0 \\ -\mathbf{1}_i\Delta_{h_j} \end{pmatrix}
\begin{tabular}{l}
$\left.\lefteqn{\phantom{\begin{matrix} 0 \\ \vdots \\ 0 \end{matrix}}}\right\}\sum_{k=1}^{j-1}m(\pi_j)$ zeroes \\
$\lefteqn{\phantom{\begin{matrix} A_i^{\pi_j} \end{matrix}}}m(\pi_j)$ rows \\
$\left.\lefteqn{\phantom{\begin{matrix} 0 \\ \vdots \\ 0 \end{matrix}}}\right\}\sum_{k=j+1}^{c}m(\pi_j)$ zeroes \\
$\left.\lefteqn{\phantom{\begin{matrix} 0 \\ \vdots \\ 0 \end{matrix}}}\right\}f(j-1)$ zeroes \\
$\left.\lefteqn{\phantom{\begin{matrix} -e_i\cdot\bs{b}_1 \\ \vdots \\ -e_i\cdot\bs{b}_f \end{matrix}}}\right\}f$ rows \\
$\left.\lefteqn{\phantom{\begin{matrix} 0 \\ \vdots \\ 0 \end{matrix}}}\right\}f(c-j)$ zeroes \\
$\lefteqn{\phantom{\begin{matrix}-\mathbf{1}_i\Delta_{h_j} \end{matrix}}}n$ rows \\
\end{tabular}
\end{equation*}
where $\mathbf{1}_i\Delta_{h_j}$ is the matrix product of the row vector with $1$ at the $i$th position and zeroes elsewhere, and $\Delta_{h_j}$.

Now, noting that \[\mathbf{1}_i\Delta_{h_j}\cdot\begin{pmatrix}g_1\\g_2\\ \vdots \\ g_n\end{pmatrix} = e_i\cdot\left(\Delta_{h_j}\begin{pmatrix} g_1\\g_2\\ \vdots \\ g_n\end{pmatrix}\right),\] we see that a $c$-tuple of words $\left(w^{(1)},\ldots,w^{(c)}\right)$ satisfies \eqref{eq:cosetcondition} for some $i$ precisely when there exists $v\in\Z^{m(\bs{\pi})+fc+n}$ such that \[L_i^j\cdot v=-B_i^{\pi_j}+e_i\cdot x_{h_jt}\] and $p_{m(\bs{\pi})}(v)=\left(\overrightarrow{w^{(1)}},\ldots,\overrightarrow{w^{(c)}}\right)\in\Z^{m(\bs{\pi})}$.
Therefore we have
\begin{equation*}
\overrightarrow{V(\bs{\pi},t)}=p_{m(\bs{\pi})}\left(\bigcap_{j=1}^c\bigcap_{i=1}^n\left\{v\in\Z^{m(\bs{\pi})+fc+n}\,\middle\vert\, L_i^j\cdot v = -B_i^{\pi_j} + e_i\cdot x_{h_jt}\right\}\right),
\end{equation*}
which is a positive polyhedral subset of $\Z^{m(\bs{\pi})}$. Now since finite unions of polyhedral sets are polyhedral, $\overrightarrow{V(\bs{\pi})}=\bigcup_{t\in T}\overrightarrow{V(\bs{\pi},t)}$ is polyhedral, which proves the Theorem.
\end{proof}

%%%%%%%%%%%%%%%%%%%%%%%%%%%%%%%%%%%%%%%%%%%%%%%%%%%%%%%%%%%%%%%%%%%%%%%%%%%%%%%%%%%%%%%%%%%%
%%%%%%%%%%%%%%%%%%%%%%%%%%%%%%%%%%%%%%%%%%%%%%%%%%%%%%%%%%%%%%%%%%%%%%%%%%%%%%%%%%%%%%%%%%%%
%%%%%%%%%%%%%%%%%%%%%%%%%%%%%%%%%%%%%%%%%%%%%%%%%%%%%%%%%%%%%%%%%%%%%%%%%%%%%%%%%%%%%%%%%%%%
%%%%%%%%%%%%%%%%%%%%%%%%%%%%%%%%%%%%%%%%%%%%%%%%%%%%%%%%%%%%%%%%%%%%%%%%%%%%%%%%%%%%%%%%%%%%
%%%%%%%%%%%%%%%%%%%%%%%%%%%%%%%%%%%%%%%%%%%%%%%%%%%%%%%%%%%%%%%%%%%%%%%%%%%%%%%%%%%%%%%%%%%%
%%%%%%%%%%%%%%%%%%%%%%%%%%%%%%%%%%%%%%%%%%%%%%%%%%%%%%%%%%%%%%%%%%%%%%%%%%%%%%%%%%%%%%%%%%%%
%%%%%%%%%%%%%%%%%%%%%%%%%%%%%%%%%%%%%%%%%%%%%%%%%%%%%%%%%%%%%%%%%%%%%%%%%%%%%%%%%%%%%%%%%%%%
%%%%%%%%%%%%%%%%%%%%%%%%%%%%%%%%%%%%%%%%%%%%%%%%%%%%%%%%%%%%%%%%%%%%%%%%%%%%%%%%%%%%%%%%%%%%
%%%%%%%%%%%%%%%%%%%%%%%%%%%%%%%%%%%%%%%%%%%%%%%%%%%%%%%%%%%%%%%%%%%%%%%%%%%%%%%%%%%%%%%%%%%%
%%%%%%%%%%%%%%%%%%%%%%%%%%%%%%%%%%%%%%%%%%%%%%%%%%%%%%%%%%%%%%%%%%%%%%%%%%%%%%%%%%%%%%%%%%%%

\section{Conjugacy Growth Series}\label{sec:conjugacy}
In this section we will prove the following.
\begin{theorem}\label{thm:conj}
 Let $G$ be a virtually abelian group, with finite generating set $S$, and weight function $\omega\colon S\rightarrow\N_+$. Then the weighted conjugacy growth series of $G$ with respect to $S$ is rational.
\end{theorem}

In order to prove this Theorem, we show that the set of conjugacy classes of a virtually abelian group can be split into an infinite collection of finite classes, and an infinite collection of infinite classes. For the finite case, \cite{Benson} gives us a way to find a minimal representative for each element of the conjugacy class, and express a full set of representatives using polyhedral sets. In the infinite case, we express each conjugacy class as a finite union of cosets of certain subgroups. Section \ref{sec:coset} gives us a way to find a minimal representative word for a given coset, and to express a full set of such representatives using polyhedral sets. We thus find a finite set of candidates for a unique minimal representative for every conjugacy class (finite or infinite). This allows us to use Lemma \ref{lem:minimaltuplerep} to extract a single such representative for each class, so that the polyhedral set description, and thus rational growth, is preserved.

As above, we assume that $G$ contains $\Z^n$ as a normal subgroup, with $[G\colon\Z^n]=d<\infty$, and we let $T$ be a choice of transversal for $G/\Z^n$ such that $1_G\in T$. Furthermore, we fix an order on $T$: \[T=\{1,t_2,t_3,\ldots,t_d\}.\]

First, we must understand the structure of conjugacy classes in virtually abelian groups. Conjugacy classes have different structure depending on whether they are inside or outside the centralizer of $\Z^n$, $C_G(\Z^n)$. Thus we consider these cases separately. Note that if one element of a coset $\Z^nt$ centralizes $\Z^n$ then $t$ must centralize $\Z^n$ and hence the whole coset is in $C_G(\Z^n)$. So both $C_G(\Z^n)$ and $G\setminus C_G(\Z^n)$ are unions of $\Z^n$-cosets.

\subsection{Conjugacy classes of elements inside the centralizer of $\Z^n$}
\begin{lemma}\label{lem:structureinside}
Let $g\in C_G(\Z^n)$. Then the conjugacy class of $g$ has size at most $d$, and is given by \[[g]=\{tgt^{-1}\mid t\in T\}=\{g,t_2gt_2^{-1}\ldots,t_dgt_d^{-1}\}.\]
\end{lemma}
\begin{proof}
Let $h\in G$. Then $hgh^{-1}=xtgt^{-1}x^{-1}$ for some $x\in\Z^n$ and $t\in T$. Since the centralizer of a normal subgroup is itself a normal subgroup, $tgt^{-1}$ centralizes $\Z^n$, and so $hgh^{-1}=tgt^{-1}$.
\end{proof}

Recall the sets $U^\pi$ of minimal-length representatives for $\pi$-patterned words, introduced in Theorem \ref{thm:Upi}. Each element of a conjugacy class has a unique minimal-weight representative, and this is contained in $U^\pi$ for some $\pi\in P$ (recall Definition \ref{def:finitepatterns}). So by Lemma \ref{lem:structureinside}, each conjugacy class in $C_G(\Z^n)$ has at most $d$ candidate words for a weight minimal representative. A $d$-tuple of candidates has a $d$-fold pattern, the $d$-dimensional vector where the entries are the patterns of the component words of the $d$-tuple. We will show that for each $d$-fold pattern in $P^d$, the corresponding set of $d$-tuples of candidate representatives forms a polyhedral set.
\begin{definition}
Fix a $d$-fold pattern $\bs{\pi}=(\pi_1,\pi_2,\ldots,\pi_d)\in P^d$ where $\overline{\pi_j}\in C_G(\Z^n)$ for each $j$, and $t_j\pi_1t_j^{-1}\in\Z^n\pi_j$ for each $2\leq j\leq d$. Let
\begin{align*}
C(\bs{\pi}):=\left\{\left(w^{(1)},\ldots,w^{(d)}\right)\in \right.&\left.U^{\pi_1}\times\cdots\times U^{\pi_d}\,\middle\vert\,\right. \\ &\left.\overline{w^{(j)}}=t_j\overline{w^{(1)}}t_j^{-1},~\text{ for each }2\leq j\leq d\right\}.
\end{align*}
\end{definition}
\begin{remark}\label{rem:candidatesinside}
Note that each tuple in $C(\bs{\pi})$ corresponds to a conjugacy class, and (by definition of the sets $U^{\pi_j}$) the weight of a conjugacy class is realised by at least one of the words in the corresponding tuple.
\end{remark}

\begin{proposition}
For each $C(\bs{\pi})$, the set $\overrightarrow{C(\bs{\pi})}\subset\N^{m(\bs{\pi})}$ is polyhedral.
\end{proposition}
\begin{proof}
Consider $w^{(1)}\in U^{\pi_1}$ and $w^{(j)}\in U^{\pi_j}$ where $t_j\overline{\pi_1}t_j^{-1}\in\Z^n\overline{\pi_j}$. By Lemma \ref{lem:detectconj}, $\overline{w^{(j)}}=t_j\overline{w^{(1)}}t_j^{-1}$ if and only if 
\begin{equation}\label{eq:conjcondition1}
A_{i,t_j}^{\pi_1}\cdot \overrightarrow{w^{(1)}}-A_i^{\pi_j}\cdot \overrightarrow{w^{(j)}}=B_i^{\pi_j}-B_{i,t_j}^{\pi_1}
\end{equation}
for each $1\leq i\leq n$. We express this using linear algebra.
Define 
\begin{equation*}
F_i^j(\bs{\pi})=\begin{pmatrix} A_{i,t_j}^{\pi_1} \\ 0 \\ \vdots \\ 0 \\ -A_i^{\pi_j} \\ 0 \\ \vdots \\ 0 \end{pmatrix}
\begin{tabular}{l}
$\lefteqn{\phantom{\begin{matrix} A_{i,t_j}^{\pi_1} \end{matrix}}}$ \small{$m(\pi_1)$ rows}\\
$\left.\lefteqn{\phantom{\begin{matrix} 0 \\ \vdots \\ 0 \end{matrix}}}\right\}$ \small{$\sum_{k=2}^{j-1}m(\pi_k)$ zeroes} \\
$\lefteqn{\phantom{\begin{matrix} -A_i^{\pi_j} \end{matrix}}}$ \small{$m(\pi_j)$ rows} \\
$\left.\lefteqn{\phantom{\begin{matrix} 0 \\ \vdots \\ 0 \end{matrix}}}\right\}$ \small{$\sum_{k=j+1}^d m(\pi_k)$ zeroes}
\end{tabular}
\end{equation*}
for each $1\leq i\leq n$, $2\leq j\leq d$. Then the vector $\left(\overrightarrow{w^{(1)}},\ldots,\overrightarrow{w^{(d)}}\right)^T$ satisfies \eqref{eq:conjcondition1} for some $j$ if and only if \[F_i^j(\bs{\pi})\cdot \begin{pmatrix}\overrightarrow{w^{(1)}} \\ \vdots \\ \overrightarrow{w^{(d)}} \end{pmatrix}=B_i^{\pi_j}-B_{i,t_j}^{\pi_1}\] for each $1\leq i \leq n$.
Thus \[\overrightarrow{C(\bs{\pi})}= \left(\overrightarrow{U^{\pi_1}}\times\cdots\times \overrightarrow{U^{\pi_d}}\right)\cap \bigcap_{j=2}^d \bigcap_{i=1}^n \left\{ \vec{z}\in\Z^{m(\bs{\pi})}\,\middle\vert\, F_i^j(\bs{\pi})\cdot\vec{z} = B_i^{\pi_j}-B_{i,t_j}^{\pi_1}\right\}.\]
This is therefore a polyhedral set.
\end{proof}

\subsection{Conjugacy classes of elements outside the centralizer of $\Z^n$}

We express the conjugacy classes in terms of certain cosets, and use the sets $U_F^\pi$ introduced in Definition \ref{def:UHpi} to find polyhedral sets of conjugacy class representatives.

\begin{definition}
	For any $\gamma\in G$, define the subgroup
	\begin{equation*}
	F(\gamma)=\{[x,\gamma]\mid x\in\Z^n\}.
	\end{equation*}
\end{definition}
Note that this is indeed a subgroup of $G$, since if $x,y\in\Z^n$, we have \[[x,\gamma][y,\gamma]=x\gamma x^{-1}\gamma^{-1}y\gamma y^{-1}\gamma^{-1}=xy\gamma x^{-1}\gamma^{-1}\gamma y^{-1}\gamma^{-1}=[xy,\gamma].\] Furthermore, since $\Z^n$ is normal, $[x,\gamma]\in\Z^n$, and so $F(\gamma)$ is a subgroup of $\Z^n$, and hence is free abelian.
\begin{remark}\label{rem:Hgroups}
Let $a\in\Z^n$ and $t\in T$. Then since $\Z^n$ is normal, $[x,at]=xatx^{-1}t^{-1}a^{-1}=xtx^{-1}t^{-1}aa^{-1}=[x,t]$. So $F(\gamma)$ depends only on the $\Z^n$-coset that $\gamma$ is contained in. Thus if $w\in W^\pi$ then $F(\overline{w})=F(\overline{\pi})$.
\end{remark}

If $A$ and $B$ are subsets of some group $G$, write $\prescript{B}{}{A}$ for the conjugate of $A$ by $B$, that is $\prescript{B}{}{A}=\{bab^{-1}\mid a\in A,b\in B\}$.
\begin{lemma}\label{lem:structureoutside}
	If $g\in G\setminus C_G(\Z^n)$ then its conjugacy class is given by a union of finitely many cosets as follows
	\begin{equation*}
	[g]=\bigcup_{t\in T}\left\{[x,tgt^{-1}]\mid x\in\Z^n\right\}tgt^{-1} = \bigcup_{t\in T}F(tgt^{-1})tgt^{-1}.
	\end{equation*}
\end{lemma}
\begin{proof}
Let $g\in G\setminus C_G(\Z^n)$, and suppose $x\in\Z^n$. We have $xgx^{-1}=xgx^{-1}g^{-1}g=[x,g]g$. Now the conjugacy class is given by
\begin{align*}
	[g]=\prescript{G}{}{\{g\}}&=\{xtg(xt)^{-1}\mid x\in\Z^n,t\in T\}=\prescript{\Z^n}{}{\{tgt^{-1}\mid t\in T\}} \\
	&=\bigcup_{t\in T}\prescript{\Z^n}{}{\{tgt^{-1}\}}   =\bigcup_{t\in T} \left\lbrace \left[x,tgt^{-1}\right]\mid x\in\Z^n\right\rbrace tgt^{-1}.
	\end{align*}
\end{proof}

\begin{definition}
Fix a $d$-fold pattern $\bs{\pi}=(\pi_1,\ldots,\pi_d)$ where $\overline{\pi_j}\in G\setminus C_G(\Z^n)$ for each $j$, and $t_j\overline{\pi_1}t_j^{-1}\in\Z^n\overline{\pi_j}$ for $2\leq j\leq d$. Define \begin{align*}
C'(\bs{\pi})=\left\{\left(w^{(1)},\ldots,w^{(d)}\right)\in\right.&\left. U_{F(\overline{\pi_1})}^{\pi_1}\times\cdots\times U_{F(\overline{\pi_d})}^{\pi_d}\,\middle\vert\, \right. \\
&\left.\overline{w^{(j)}}\in\prescript{\Z^nt_j}{}{\left\{\overline{w^{(1)}}\right\}},~2\leq j\leq d\right\},
\end{align*}
where $U_{F(\overline{\pi_i})}^{\pi_i}$ is as in Definition \ref{def:UHpi}.
That is, $C'(\bs{\pi})$ is the set of $d$-tuples of words where each $w^{(j)}$ is the unique minimal representative for its $F(\overline{\pi_j})$-coset, and $\overline{w^{(j)}}$ is conjugate to $\overline{w^{(1)}}$ via an element of $\Z^nt_j$. 
\end{definition}

\begin{proposition}\label{prop:candidatesoutside}
Each element of $C'(\bs{\pi})$ consists of a $d$-tuple of words representing elements of the same conjugacy class. Furthermore, the weight of the conjugacy class is realised by the word(s) of smallest weight in the $d$-tuple.
\end{proposition}
\begin{proof}
Let $\left(w^{(1)},w^{(2)},\ldots,w^{(d)}\right)\in C'(\bs{\pi})$. From the definition of $C'(\bs{\pi})$, each $\overline{w^{(j)}}$ is conjugate to $\overline{w^{(1)}}$, so each component word represents an element of the same conjugacy class. Now from Lemma \ref{lem:structureoutside}, we see that each word represents one of the finite number of cosets that make up the corresponding conjugacy class. In fact, since each $w^{(j)}$ is contained in $U_{F(\overline{\pi_j})}$, each component word is a minimal-weight representative for the coset. Therefore the minimal-weight representative(s) for the conjugacy class must be contained in $\{w^{(1)},w^{(2)},\ldots,w^{(d)}\}$.
\end{proof}

\begin{proposition}
The set $\overrightarrow{C'(\bs{\pi})}\in\N^{m(\bs{\pi})}$ is a polyhedral set.
\end{proposition}

\begin{proof}
We have
\begin{align*}
\prescript{\Z^nt_j}{}{\left\{\overline{w^{(1)}}\right\}}=F\left(t_j\overline{w^{(1)}}t_j^{-1}\right)t_j\overline{w^{(1)}}t_j^{-1}=F(\overline{\pi_j})t_j\overline{w^{(1)}}t_j^{-1}.
\end{align*}
by Remark \ref{rem:Hgroups}. Thus $\overline{w^{(j)}}\in\prescript{\Z^nt_j}{}{\left\{\overline{w^{(1)}}\right\}}$ if and only if there exists $y\in F(\overline{\pi_j})$ with
\begin{equation}\label{eq:conjcondition2}
\overline{w^{(j)}}=yt_j\overline{w^{(1)}}t_j^{-1}.
\end{equation}
Write $f_j$ for the rank of the free abelian group $F(\overline{\pi_j})\subset\Z^n$. Let $\{\bs{b}_1^{\pi_j},\ldots,\bs{b}_{f_j}^{\pi_j}\}$ be a choice of basis for $F(\overline{\pi_j})$. Then there exists $y\in F(\overline{\pi_j})$ satisfying \eqref{eq:conjcondition2} if and only if there exist integers $a_1,\ldots,a_{f_j}$ with
\begin{equation*}
\overline{w^{(j)}}=\sum_{k=1}^{f_j}a_k\bs{b}_k^{\pi_j}t_j\overline{w^{(1)}}t_j^{-1}.
\end{equation*}
Expanding using identities \eqref{eq:wbar} and \eqref{eq:wbarconj} gives
\begin{align*}
\left[\begin{pmatrix} A_1^{\pi_j}\cdot\overrightarrow{w^{(j)}} \\ A_2^{\pi_j}\cdot\overrightarrow{w^{(j)}} \\ \vdots \\ A_n^{\pi_j}\cdot\overrightarrow{w^{(j)}} \end{pmatrix} + \begin{pmatrix}B_1^{\pi_j} \\ B_2^{\pi_j} \\ \vdots \\ B_n^{\pi_j}\end{pmatrix}\right]t_{\pi_j}= 
\left(\sum_{k=1}^{f_j}a_k\bs{b}_k^{\pi_j}\right)\left[\begin{pmatrix} A_{1,t_j}^{\pi_1}\cdot\overrightarrow{w^{(1)}} \\ A_{2,t_j}^{\pi_1}\cdot\overrightarrow{w^{(1)}} \\ \vdots \\ A_{n,t_j}^{\pi_1}\cdot\overrightarrow{w^{(1)}}\end{pmatrix}+\begin{pmatrix}B_{1,t_j}^{\pi_1} \\ B_{2,t_j}^{\pi_1} \\ \vdots \\ B_{n,t_j}^{\pi_1}\end{pmatrix}\right]t_{\pi_j},
\end{align*}
or equivalently,
\begin{equation}\label{eq:conjcondition3}
A_{i,t_j}^{\pi_1}\cdot\overrightarrow{w^{(1)}} - A_i^{\pi_j}\cdot\overrightarrow{w^{(j)}} + e_i\cdot\sum_{k=1}^{f_j}a_k\bs{b}_k^{\pi_j} = B_i^{\pi_j} - B_{i,t_j}^{\pi_1}
\end{equation}
for each $1\leq i\leq n$. We express this using linear algebra.

For each $1\leq i\leq n$ and $2\leq j\leq d$, consider the vectors
\begin{equation*}
M_i^j(\bs{\pi}) = \begin{pmatrix} A_{i,t_j}^{\pi_1} \\ 0 \\ \vdots \\ 0 \\ -A_i^{\pi_j} \\ 0 \\ \vdots \\ 0 \end{pmatrix} 
\begin{tabular}{l}
$\lefteqn{\phantom{\begin{matrix} A_{i,t_j}^{\pi_1} \end{matrix}}}$ \small{$m(\pi_1)$ rows}\\
$\left.\lefteqn{\phantom{\begin{matrix} 0 \\ \vdots \\ 0 \end{matrix}}}\right\}$ \small{$\sum_{k=2}^{j-1}m(\pi_k)$ zeroes}\\
$\lefteqn{\phantom{\begin{matrix} -A_i^{\pi_j} \end{matrix}}}$ \small{$m(\pi_j)$ rows} \\
$\left.\lefteqn{\phantom{\begin{matrix} 0 \\ \vdots \\ 0 \end{matrix}}}\right\}$ \small{$\sum_{k=j+1}^d m(\pi_k)$ zeroes}\\
\end{tabular}
\end{equation*}
and
\begin{equation*}
N_i^j(\bs{\pi}) = \begin{pmatrix} 0 \\ \vdots \\ 0 \\ e_i\cdot\bs{b}_1^{\pi_j}\\ e_i\cdot\bs{b}_2^{\pi_j} \\ \vdots \\ e_i\cdot\bs{b}_{f_j}^{\pi_j} \\ 0 \\ \vdots \\ 0 \end{pmatrix}
\begin{tabular}{l}
$\left.\lefteqn{\phantom{\begin{matrix} 0 \\ \vdots \\ 0 \end{matrix}}}\right\}$ \small{$\sum_{k=2}^{j-1}f_k$ zeroes} \\
$\left.\lefteqn{\phantom{\begin{matrix} e_i\cdot\bs{b}_1^{\pi_j}\\ e_i\cdot\bs{b}_2^{\pi_j} \\ \vdots \\ e_i\cdot\bs{b}_{f_j}^{\pi_j} \end{matrix}}}\right\}$ \small{$f_j$ rows} \\
$\left.\lefteqn{\phantom{\begin{matrix} 0 \\ \vdots \\ 0 \end{matrix}}}\right\}$ \small{$\sum_{k=j+1}^d f_k$ zeroes}
\end{tabular}
\end{equation*}
Let $f=\sum_{j=1}^d f_j$, i.e. the sum of the ranks of the free abelian subgroups $F(\overline{\pi_j})$, and hence the dimension of $N_i^j(\bs{\pi})$. Now by equation \eqref{eq:conjcondition3}, $\overline{w^{(j)}}\in\prescript{\Z^nt_j}{}{\left\{\overline{w^{(1)}}\right\}}$ precisely when there exist integers $a_1,a_2,\ldots a_f$ such that 
\begin{equation*}
M_i^j(\bs{\pi})\cdot\begin{pmatrix} \overrightarrow{w^{(1)}} \\ \overrightarrow{w^{(2)}} \\ \vdots \\ \overrightarrow{w^{(d)}}\end{pmatrix} + N_i^j(\bs{\pi})\cdot\begin{pmatrix} a_1 \\ a_2 \\ \vdots \\ a_f\end{pmatrix} = B_i^{\pi_j}-B_{i,t_j}^{\pi_1}
\end{equation*}
for each $1\leq i\leq n$ and $2\leq j\leq d$. We see that each of these identities defines an elementary set if we rewrite it as follows:
\begin{equation*}
\left\{\vec{z}\in\N^{m(\bs{\pi})+f}\,\middle\vert\, \begin{pmatrix} \vert \\ M_i^j(\bs{\pi}) \\ \vert \\ \vert \\ N_i^j(\bs{\pi}) \\ \vert \end{pmatrix}\cdot\vec{z} = B_i^{\pi_j} - B_{i,t_j}^{\pi_1} \right\}.
\end{equation*}
Taking the intersection of each elementary set, discarding the vector $(a_1,\ldots,a_r)$, and intersecting with the cartesian product of polyhedral sets $U_{F(\overline{\pi_1})}\times\cdots\times U_{F(\overline{\pi_d})}$, allows us to express the $m(\bs{\pi})$-dimensional vectors corresponding to $C'(\bs{\pi})$ as a polyhedral set:
\begin{align*}
\overrightarrow{C'(\bs{\pi})}=\left(U_{F(\overline{\pi_1})}\times\cdots\times U_{F(\overline{\pi_d})}\right)&\cap \\
\bigcap_{j=2}^d\bigcap_{i=1}^n p_{m(\bs{\pi})}&\left\{\vec{z}\in\N^{m(\bs{\pi})+r}\,\middle\vert\, \begin{pmatrix} \vert \\ M_i^j(\bs{\pi}) \\ \vert \\ \vert \\ N_i^j(\bs{\pi}) \\ \vert \end{pmatrix}\cdot\vec{z} = B_i^{\pi_j} - B_{i,t_j}^{\pi_1} \right\}
\end{align*}
where, as before, $p_k$ denotes projection onto the first $k$ coordinates.
\end{proof}

We are now ready to prove Theorem \ref{thm:conj}.
\begin{proof}[Proof of Theorem \ref{thm:conj}]
Each conjugacy class in $G$ has a $d$-tuple of candidates for a weight-minimal representative (see Remark \ref{rem:candidatesinside} and Proposition \ref{prop:candidatesoutside}), given by an element of $C(\bs{\pi})$ or $C'(\bs{\pi})$ for an appropriate $d$-fold pattern $\bs{\pi}$. The result will follow from the following claim:

There exists a finite set of $d$-fold patterns, $R$, so that:
\begin{enumerate}
\item the candidate representatives for every conjugacy class in $G$ (as $d$-tuples of words) have a $d$-fold pattern in $R$, and\label{item:Asufficient}
\item no conjugacy class is represented by more than one $d$-fold pattern in $R$.\label{item:Bunique}
\end{enumerate}
For each $\bs{\pi}\in R$, Lemma \ref{lem:minimaltuplerep} yields a set $\cL_{\bs{\pi}}\in\wt{S}^*$ of unique, minimal-weight representatives for the tuples of $C(\bs{\pi})$, or $C'(\bs{\pi})$. It follows from the above claim that $\bigcup_{\bs{\pi}\in R}\cL_{\bs{\pi}}$ is a finite disjoint union of sets, forming a language of unique minimal-weight representatives for the conjugacy classes of $G$. Since each $\cL_{\bs{\pi}}$ has rational growth series, we conclude that $G$ has rational conjugacy growth series.

Now we prove the claim. If $P$ is the finite set of patterns (with respect to $\wt{S}^*$) providing minimal weight representatives for each element of $G$ (as per Definition \ref{def:finitepatterns}), consider the set of ordered $d$-tuples $(\pi_1,\pi_2,\ldots,\pi_d)$ of elements of $P$, with the condition that $t_i\pi_1 t_i^{-1}\in\Z^n\pi_j$. For any set of such $d$-tuples which are permutations of each other, choose only one (arbitrarily), and discard the others. Call the resulting reduced set of $d$-fold patterns $R$. This is clearly a finite set, and is sufficient to represent all $d$-tuples of elements of $G$. This proves part \eqref{item:Asufficient}.

To see part \eqref{item:Bunique}, note that the candidates are uniquely determined (either the unique weight-minimal representatives for each element, in the $C_G(\Z^n)$ case, or the unique weight-minimal representatives for each coset component, in the $G\setminus C_G(\Z^n)$ case). A tuple of candidates uniquely determines a $d$-fold pattern in $R$ (since we have removed permutations). Thus the claim holds, and the theorem follows.
\end{proof}

%%%%%%%%%%%%%%%%%%%%%%%%%%%%%%%%%%%%%%%%%%%%%%%%%%%%%%%%%%%%%%%%%%%%%%%%%%%%
%%%%%%%%%%%%%%%%%%%%%%%%%%%%%%%%%%%%%%%%%%%%%%%%%%%%%%%%%%%%%%%%%%%%%%%%%%%%
\section*{Acknowledgments}
%%%%%%%%%%%%%%%%%%%%%%%%%%%%%%%%%%%%%%%%%%%%%%%%%%%%%%%%%%%%%%%%%%%%%%%%%%%%
%%%%%%%%%%%%%%%%%%%%%%%%%%%%%%%%%%%%%%%%%%%%%%%%%%%%%%%%%%%%%%%%%%%%%%%%%%%%

The author would primarily like to thank Laura Ciobanu for an immeasurable amount of mathematical discussion and writing advice. Thanks are also due to Turbo Ho for a helpful discussion, and Jim Howie for useful comments on a draft of this paper.  

%%%%%%%%%%%%%%%%%%%%%%%%%%%%%%%%%%%%%%%%%%%%%%%%%%%%%%%%%%%%%%%%%%%%%%%%%%%%

\bibliography{references}{}
\bibliographystyle{amsplain}

\end{document}